\documentclass[9pt]{article}%
\usepackage{bbm}
\usepackage{amssymb}
\usepackage{amsmath}
\usepackage{amssymb}
\usepackage{mathrsfs}
\usepackage{amsfonts}
\usepackage{graphicx}
\usepackage{color}
\usepackage{mathrsfs}
\usepackage{graphicx}
\usepackage{easybmat}
\usepackage{epsfig}
\usepackage{epstopdf}
\usepackage{graphics,graphicx,amssymb,amsmath,verbatim}
\usepackage{amssymb}
\usepackage{mathrsfs}
\usepackage{amsfonts}
\usepackage{amsmath}
\usepackage{graphicx}
\usepackage{easybmat}
\usepackage{color}
\usepackage{mathrsfs}
\usepackage{leftidx}
\usepackage{easybmat}
\usepackage{wrapfig}
\usepackage{authblk}
\usepackage{booktabs}%
\providecommand{\U}[1]{\protect \rule{.1in}{.1in}}
\newenvironment{proof}[1][Proof]{\noindent \textbf{#1.} }{\  \rule{0.5em}{0.5em}}
\newtheorem{remark}{Remark}
\newtheorem{theorem}{Theorem}
\newtheorem{lemma}{Lemma}

\newtheorem{example}{Example}
\newtheorem{assumption}{Assumption}
\allowdisplaybreaks[4]
\begin{document}

\title{Stabilization of linear systems with multiple unknown time-varying input
delays by linear time-varying feedback}
\author{Bin Zhou, Kai Zhang \thanks{The authors are with the Center for Control Theory
and Guidance Technology, Harbin Institute of Technology, Harbin, 150001,
China, and also with National Key Laboratory of Complex System Control and Intelligent Agent Cooperation, Harbin 150001, China.  Email: \texttt{
binzhou@hit.edu.cn, kaizhang2023@hit.edu.cn}.}}
\date{}
\maketitle

\begin{abstract}
This paper addresses the stabilization of linear systems with multiple
time-varying input delays. In scenarios where neither the exact delays
information nor their bound is known, we propose a class of linear
time-varying state feedback controllers by using the solution to a parametric
Lyapunov equation (PLE). By leveraging the properties of the solution to the
PLE and constructing a time-varying Lyapunov-Krasovskii-like functional, we
prove that (the zero solution of) the closed-loop system is asymptotically
stable. Furthermore, this result is extended to the observer-based output
feedback case. The notable characteristic of these controllers is their
utilization of linear time-varying gains. Furthermore, they are designed
entirely independent of any knowledge of the time delays, resulting in
controllers that are exceedingly easy to implement. Finally, a numerical
example demonstrates the effectiveness of the proposed approaches.


\end{abstract}

\section{Introduction}

Time delay, prevalent across engineering systems such as nuclear reactors,
hydraulic systems, manufacturing processes, digital control systems, and
remote control setups, is widely recognized as a significant contributor to
performance degradation and instability in control systems \cite{lixiaodi2021}%
, \cite{zhangjin2022}, \cite{zhuyang2020}. Hence, control of time-delay
systems \cite{Basin2008}, \cite{Fridman2020}, \cite{lihanfeng2021},
particularly those concerning asymptotic stability analysis and stabilization
\cite{huangyibo2021}, has been a focal point for research over decades.
Extensive literature has investigated various types of time-delay systems,
leading to a rich collection of findings and insights. Generally speaking, for
stability analysis, tools are available in both the frequency domain
\cite{gukeqi2003}, \cite{lam1990}, \cite{Michiels2005},
\cite{Michiels2005auto}, \cite{Olgacand2014} and the time domain
\cite{xuxiang2020}, \cite{yuxin2023}, each possessing equivalent power in
certain respects. Specifically, within the time domain, one of the most
efficient methods for stability analysis is arguably the Lyapunov-Krasovskii
functional-based approach, as exemplified by studies \cite{Chaillet2022},
\cite{Fridman2006}, \cite{Selivanov2019}, \cite{wang2014} and
\cite{zhangbaoyong2015}. This method, along with the Razumikhin theorem-based
approach \cite{Efimov2020}, \cite{Nekhoroshikh2023}, \cite{Zhangjunfeng2021},
aims to identify a positive-definite functional whose derivative along the
trajectories of the time-delay system is negative. The outcomes derived from
these methods can often be reformulated into linear matrix inequalities
(LMIs), which can be efficiently solved numerically.

For stabilizing time-delay systems, two primary approaches are commonly
employed. The first approach centers on locating finite-dimensional
controllers, leading to an infinite-dimensional closed-loop system. Here are a
few brief examples to illustrate this method. By utilizing the LMI-based
robust controller design methodology, a time-delay independent control
strategy was developed in \cite{yuedong2004}. This approach ensures system
stability while also enabling the closed-loop system to withstand the
uncertainties associated with time delay within certain predefined limits. For
a class of linear systems with time-varying state and input delays, a
so-called integral inequality method was proposed in \cite{zhangxianming2005}
to tackle the delay-dependent stabilization problem. However, in the setup of
finite-dimensional controllers, stability typically hinges on the delay,
making it challenging to guarantee stability for arbitrarily large delays. In
contrast, the second approach focuses on designing infinite-dimensional
controllers. These controllers aim to emulate the behavior of a delay-free
system within the closed-loop system, thereby ensuring stability regardless of
the delay. One highly efficient method falling within this category is
predictor feedback. This method is particularly effective for input delay
systems, even those that are unstable. However, as observed by
\cite{Krstic2010TAC}, existing studies predominantly address systems with
constant delays, offering limited insights into scenarios involving
time-varying input delays. While Artstein initially proposed predictor
feedback for time-varying input delay systems in \cite{Artstein1982}, the
intricacies of its design remain under-specified, primarily due to its focus
on plants with temporal variations. In a recent development, the author in
\cite{Krstic2010TAC} demonstrated exponential asymptotic stability for
feedback systems experiencing time-varying input delays. This achievement
involved constructing a Lyapunov-Krasovskii functional using a backstepping
transformation featuring time-varying kernels. However, the predictor feedback
employs the infinite-dimensional term, which encounters difficulties in actual
implementation \cite{Fridman2014book}, \cite{zhou14book}. By carefully
ignoring the infinite-dimensional terms, authors in \cite{Zhou2012},
\cite{zhou2009}, \cite{zhou2014TAC} proposed a truncated predictor feedback
controller which only contains a finite dimensional term and is easy to implement.

It is essential to highlight that the aforementioned results necessitate
precise information regarding the time delay \cite{Krstic2010TAC},
\cite{Zhou2012} or at least knowledge of its upper or lower bounds
\cite{yuedong2004}, \cite{zhangxianming2005}, \cite{zhou2009}. Consequently,
these findings may not be applicable in scenarios where the time delay is
unknown, and neither its upper nor lower bounds are ascertainable. In fact,
addressing stabilization problems in such contexts presents a significant
challenge, reflected in the scarcity of literature reports on related results.
This challenge arises because the parameters within the feedback gain are
contingent on knowledge of the time delays and/or their bounds. Without this
critical information, the design of state feedback lacks a solid foundation.
In order to solve this problem, one feasible approach is to design a
time-varying controller by the concept of adaptive control. In \cite{wei2019},
the delay-independent truncated predictor feedback controller proposed in
\cite{zhou2009} was generalized to linear systems with input delay where
neither the exact delay information nor its upper bound is known. This
controller, featuring two update laws with a switching mechanism, aims to
regulate both the states and the control input of the closed-loop system to
zero. However, the effectiveness of this method has only been demonstrated for
single constant input delay, leaving its performance in scenarios involving
multiple time-varying input delays yet to be verified. Additionally, the
controller in \cite{wei2019} is nonlinear and requires online solution of
differential equations, which imposes computational burden. That is to say,
even with the multitude of available control strategies, stabilizing linear
systems with multiple time-varying input delays using a feedback law
independent of any knowledge of delay remains an unresolved challenge. This
challenge is particularly pronounced in scenarios where the controller is
limited to be linear.

Motivated by the preceding statement, this paper focuses on investigating the
stabilization problem of a class of linear systems with multiple time-varying
input delays, where both the exact delays information and their bounds are
unknown. We introduce a novel time-varying linear state feedback controller
based on the solution to a parametric Lyapunov equation (PLE). Unlike existing
approaches \cite{wei2019} that require online solution of differential
equations, our proposed controllers feature pre-designed time-varying
parameters off-line, substantially alleviating computational burden. In order
to establish the stability of the closed-loop system, we further investigate
the properties of the solution to the PLE. Leveraging the properties along
with a time-varying Lyapunov-Krasovskii-like functional, we validate that both
the states and inputs of the closed-loop system can asymptotically approach to
zero. Our results extend existing findings in three significant aspects:
moving from scenarios with known delay information \cite{weiyusheng2019SCL},
\cite{Zhou2012}, \cite{zhou2009}, \cite{zhou2014TAC}, \cite{zhou14book} to
those with unknown delay information, transitioning from constant
single-input-delay systems \cite{wei2019}, \cite{weiyusheng2019SCL} to
time-varying multi-input-delays systems, and shifting from nonlinear
controllers \cite{wei2019} to linear controllers. We also extend the state
feedback to the observer-based output feedback case. Finally, a numerical
example is provided to demonstrate the effectiveness of the proposed approaches.

\textit{Notation:} Throughout this paper, the following conventions are
adopted for the matrix $A$: $A^{\mathrm{T}}$ denotes its transpose,
$\left \Vert A\right \Vert $ represents its Euclidean norm, $\lambda(A) $ stands
for its eigenvalue set, $\lambda_{i}(A)$ signifies the $i$-th eigenvalue,
$\mathrm{Re}\{ \lambda_{i}(A)\} $ denotes the real part of the $i$-th
eigenvalue, and $\mathrm{tr}(A)$ denotes its trace. Furthermore, $\phi(A)$
represents $\min_{i=1,2,\ldots n}\{ \mathrm{Re}(\lambda_{i}(A))\} $,
indicating the minimum real part of the eigenvalues of $A\in \mathbf{R}%
^{n\times n}$. Additionally, $\alpha(A)$ denotes the maximal order of the
Jordan blocks associated with an eigenvalue $\lambda_{i}(A)$ such that
$\mathrm{Re}( \lambda_{i}(A))=\phi(A)$, ensuring that $\alpha(A)\geq1$. If $A$
is symmetric, then $\lambda_{\mathrm{\min}}(A)$ and $\lambda_{\mathrm{\max}%
}(A)$ respectively represent its minimal and maximal eigenvalues, and $A>0$
denotes that $A$ is a positive definite matrix. For any matrix $A=
[a_{1},a_{2},\ldots,a_{n}] \in \mathbf{R}^{m\times n}$, the stretching function
is defined as $\mathrm{vec}(A) =[a_{1}^{\mathrm{T}},a_{2}^{\mathrm{T}}%
,\ldots,a_{n}^{\mathrm{T}}]^{\mathrm{T}}$. For two integers $k_{1}$ and
$k_{2}$ with $k_{1}\leq k_{2}$, $\mathbf{I}_{k_{1}}^{k_{2}}$ denotes $\{
k_{1},k_{1}+1,\ldots,k_{2} \} $.

\section{Problem Formulation and Preliminaries}

\subsection{Problem Formulation}

Consider the following linear system with multiple time-varying input delays
\begin{equation}
\left \{
\begin{array}
[c]{l}%
\dot{x}(t)=Ax(t)+\sum \limits_{i=1}^{q}B_{i}u(t-\tau_{i}(t)),\\
y(t) =Cx(t) ,
\end{array}
\right.  \label{sys}%
\end{equation}
where $x\in \mathbf{R}^{n}$, $u\in \mathbf{R}^{m}$ and $y\in \mathbf{R}^{p}$ are
the state, input and output, respectively, $A\in \mathbf{R}^{n\times n}$,
$B_{i}\in \mathbf{R}^{n\times m},i\in \mathbf{I}_{1}^{q}$, and $C\in
\mathbf{R}^{p\times n}$ are constant matrices, $\tau_{i}(t),i\in \mathbf{I}%
_{1}^{q}$ are the time-varying input delays. Stabilization of system
(\ref{sys}) stands out as a pivotal concern within the time-delay control
systems, garnering significant attention. Notably, the method for designing
controllers in scenarios where $\tau_{i}(t),i\in \mathbf{I}_{1}^{q}$ are large
yet bounded was elucidated in \cite{zhou14book}. Denote
\[
B=B_{1}+B_{2}+\cdots+B_{q}.
\]
Then, the derivation of this outcome is underpinned by a set of specific assumptions.

\begin{assumption}
\label{ass1}The matrix pair $(A,B)$ is controllable and all the eigenvalues of
$A$ are zeros.
\end{assumption}

\begin{remark}
The condition outlined in Assumption \ref{ass1} can be extended to scenarios
where $(A,B)$ is stabilizable, and the eigenvalues of matrix $A$ are located
at the origin or in the left half-plane. Given that stable eigenvalues of $A$
do not impact the stabilizability of the system, we choose to impose
Assumption \ref{ass1} ultimately. It is worth noting that the assumptions on
$A$ presented in this paper align with those in \cite{wei2019}. However, our
method can effectively manage multiple time-varying input delays by linear
time-varying feedback, a characteristic not covered in \cite{wei2019}.
\end{remark}

\begin{assumption}
\label{ass2}There exists a known constant $\bar{\tau}>0$ such that $0\leq
\tau_{i}(t)\leq \bar{\tau},\ i\in \mathbf{I}_{1}^{q},\forall t\geq0.$
\end{assumption}

Let the initial condition of system (\ref{sys}) be $x(0)=x_{0},u(\theta
)=u_{0}(\theta), \theta \in[-\tau(0),0)$. The following result was proven in
\cite{zhou14book}.

\begin{lemma}
\label{lemma0} Assume that $(A,B) $ satisfies Assumption \ref{ass1} and the
delays $\tau_{i}(t),\ i\in \mathbf{I}_{1}^{q}$ satisfy Assumption \ref{ass2}.
Let $P(\gamma) $ be the unique positive definite solution to the PLE
\cite{zhou14book}
\begin{equation}
A^{\mathrm{T}}P+PA-PBB^{\mathrm{T}}P=-\gamma P. \label{ple}%
\end{equation}
Then system (\ref{sys}) can be stabilized by the following linear state
feedback%
\begin{equation}
u(t)=-B^{\mathrm{T}}P(\gamma)x(t),\quad \forall \gamma \in(0,\gamma^{\ast}]
,\quad \forall t \geq0, \label{ss}%
\end{equation}
where $\gamma^{\ast}\in(0, 1/((1+\sqrt{3})n\sqrt{n}\bar{\tau})) >0$ is a
constant depending on $\bar{\tau}$.
\end{lemma}

It can be seen that the stability of the closed-loop system can only be
ensured after determining the range of $\gamma$, which relies on knowing the
upper bound $\bar{\tau}$ of the time delays. In other words, if we do not know
the upper bound $\bar{\tau}$, we cannot choose an appropriate $\gamma$ to
ensure the stability of the closed-loop system. In order to solve this
problem, in this paper we will design a time-varying $\gamma$. The basic
design idea is to gradually reduce $\gamma(t)$ to ensure that condition
(\ref{ss}) is met, thereby ensuring the stability of the closed-loop system.
Since we have no prior information about $\bar{\tau}$, $\gamma(t)$ must always
decrease (namely, $\dot{\gamma}(t)<0$), that is, it decreases to $0$ as time
$t$ approaches to infinity. Obviously, this will bring about another problem,
that is, the control gain $B^{\mathrm{T}}P(\gamma)$ will also tend to zero as
the time $t$ approaches to infinity. In such case, the system tends to operate
in an open loop. Therefore, we must design a suitable function $\gamma(t)$ and
strictly analyze the stability of the closed-loop system. This is the first
key problem to be solved in this paper.

If only $y(t)$ is available, we are interested in designing an observer-based
controller to asymptotically stabilize system (\ref{sys}). When $\tau_{i}(t),
\ i\in \mathbf{I}_{1}^{q}$ are known, it is intuitive to devise an
observer-based controller described by%
\begin{equation}
\left \{
\begin{array}
[c]{l}%
\dot{\xi}(t)=A\xi(t)+ \sum \limits_{i=1}^{q}B_{i}u(t- \tau_{i}(t))+L(y(t)-C\xi
(t)) ,\\
u(t)=-{B^{\mathrm{T}}P(\gamma(t))}\xi(t).
\end{array}
\right.  \label{eq4.0}%
\end{equation}
However, as $\tau_{i}(t), \ i\in \mathbf{I}_{1}^{q}$ are unknown, the term
$\mathit{\Sigma}_{i=1}^{q}B_{i}u(t-\tau_{i}(t))$ is unavailable. Consequently,
the controller (\ref{eq4.0}) is unsuitable for addressing the current problem.
Therefore, the design of a suitable observer-based controller becomes another
key problem to be addressed in this paper.

\subsection{Properties of Parametric Lyapunov Equations}

In this subsection, we collect some key properties of the PLE (\ref{ple}) for
our future use.

\begin{lemma}
\label{lemma1}Let $(A,B) $ be controllable.

\begin{itemize}
\item The PLE (\ref{ple}) has a (unique) positive definite solution
$P(\gamma)$ if and only if%
\begin{equation}
\gamma>-2\phi(A) . \label{plePP1}%
\end{equation}
Moreover, $P(\gamma)$ is given by $P(\gamma)=W^{-1}(\gamma) $, which satisfies
the Lyapunov equation%
\begin{equation}
\left(  A+\frac{1}{2}\gamma I_{n}\right)  W+W\left(  A+\frac{1}{2}\gamma
I_{n}\right)  ^{\mathrm{T}}=BB^{\mathrm{T}}. \label{2}%
\end{equation}

\item Let $\gamma$ satisfy (\ref{plePP1}) and denote $\pi(\gamma)
=2\mathrm{tr}(A) +n\gamma.$ Then $P(\gamma)$ has the following properties%
\begin{align}
\frac{\mathrm{d}P(\gamma)}{\mathrm{d}\gamma}  &  >0, \quad \mathrm{tr}%
(B^{\mathrm{T}}PB) =\pi(\gamma) ,\nonumber \\
\mathrm{tr}(PAP^{-1}A^{\mathrm{T}})  &  \leq \frac{n-1}{2n}\pi^{2}(\gamma)
+\frac{2}{n}\mathrm{tr}^{2}(A) -\mathrm{tr}(A^{2}) . \label{plePP6}%
\end{align}

\item If all eigenvalues of $A$ are identical and purely real, then%
\begin{equation}
\frac{\mathrm{d}P(\gamma)}{\mathrm{d}\gamma}\leq \frac{\delta_{\mathrm{c}%
}P(\gamma)}{\pi(\gamma) },\  \forall \gamma>-2\phi( A) , \label{plePP3}%
\end{equation}
where $\delta_{\mathrm{c}}$ is a constant independent of $\gamma$ and can be
determined by%
\begin{equation}
\delta_{\mathrm{c}}=\sup_{\gamma>-2\phi(A) }\{ \delta_{\mathrm{c}}^{\gamma
}\},\  \delta_{\mathrm{c}}^{\gamma}=\pi( \gamma) \lambda_{\max}(W^{-\frac{1}%
{2}}UW^{-\frac{1}{2}}) , \label{plePP7}%
\end{equation}
where $U=U(\gamma) $ is the unique positive definite solution to the Lyapunov
equation%
\begin{equation}
\left(  A+\frac{1}{2}\gamma I_{n}\right)  U+U\left(  A+\frac{1}{2}\gamma
I_{n}\right)  ^{\mathrm{T}}=W. \label{plePP8}%
\end{equation}

\end{itemize}
\end{lemma}

The first two properties in Lemma \ref{lemma1} are referenced from
\cite{zhou14book}. The proof for the final property in Lemma \ref{lemma1} will
be provided in Appendix A1.

One may wonder whether (\ref{plePP3}), which is the key property for solving
the problem in this paper, is also satisfied in cases where the eigenvalues of
$A$ are not identical. We use the following example to show that this is not true.

\begin{example}
Consider a special system with%
\[
A=\left[
\begin{array}
[c]{cc}%
0 & 1\\
0 & \alpha
\end{array}
\right]  ,\ B=\left[
\begin{array}
[c]{c}%
0\\
1
\end{array}
\right]  ,
\]
in which $\alpha \geq0$ is a constant. It is evident that $\phi(A)=0$.
Furthermore, $\lambda(A)=\{0\}$ when $\alpha=0$ and $\lambda(A)=\{0,\alpha \}$
when $\alpha>0$. By (\ref{ple}), it can be computed that%
\[
P(\gamma)=\left[
\begin{array}
[c]{cc}%
\gamma(\gamma+\alpha)^{2} & \gamma(\gamma+\alpha)\\
\gamma(\gamma+\alpha) & 2(\gamma+\alpha)
\end{array}
\right]  ,\ \frac{\mathrm{d}P(\gamma)}{\mathrm{d}\gamma}=\left[
\begin{array}
[c]{cc}%
(3\gamma+\alpha)(\gamma+\alpha) & 2\gamma+\alpha \\
2\gamma+\alpha & 2
\end{array}
\right]  .
\]
Furthermore, by (\ref{plePP8}), or equivalently%
\[
\delta_{\mathrm{c}}=\sup_{\gamma>-2\phi(A)}\{ \delta_{\mathrm{c}}^{\gamma
}\},\delta_{\mathrm{c}}^{\gamma}=\pi(\gamma)\lambda_{\max}\left(  P^{-\frac
{1}{2}}\frac{\mathrm{d}P(\gamma)}{\mathrm{d}\gamma}P^{-\frac{1}{2}}\right)  ,
\]
we deduce that%
\[
\delta_{\mathrm{c}}=\sup_{\gamma>0}\left \{  \frac{2(\sqrt{\varsigma(\gamma
)}+\alpha+\gamma)\sqrt{\varsigma(\gamma)}}{\gamma(\gamma+2\alpha)}\right \}  ,
\]
where $\varsigma(\gamma)=\alpha^{2}+4\alpha \gamma+2\gamma^{2}$. Clearly, if
$\alpha=0$, then $\delta_{\mathrm{c}}=2(2+\sqrt{2})$ which is bounded, while
if $\alpha>0$ then $\delta_{\mathrm{c}}\rightarrow \infty$ as $\gamma
\rightarrow0^{+}$.
\end{example}

As an extension of Lemma \ref{lemma1}, properties of $P(\gamma)$ when $A$
satisfies Assumption \ref{ass1} are detailed below. These properties will
serve as fundamental elements extensively utilized in this paper.

\begin{lemma}
\label{lemma2}Assume that $(A,B)$ satisfies Assumption \ref{ass1}. Then, for
any $\gamma>0,$ the PLE (\ref{ple}) has a unique positive definite solution
$P(\gamma)$. In addition, the unique solution $P(\gamma)$ possesses the
following properties%
\begin{align}
\mathrm{tr}(B^{\mathrm{T}}P(\gamma)B)  &  =n\gamma,\label{pleP1}\\
P(\gamma)BB^{\mathrm{T}}P(\gamma)  &  \leq n\gamma P(\gamma),\label{pleP2}\\
A^{\mathrm{T}}P(\gamma)A  &  \leq3n^{2}\gamma^{2}P(\gamma) ,\label{pleP3}\\
\frac{\mathrm{d}P(\gamma)}{\mathrm{d}\gamma}BB^{\mathrm{T}}\frac
{\mathrm{d}P(\gamma)}{\mathrm{d}\gamma}  &  \leq n\frac{\mathrm{d}P(\gamma
)}{\mathrm{d}\gamma},\label{pleP4}\\
\frac{P(\gamma)}{n\gamma}\leq \frac{\mathrm{d}P(\gamma)}{\mathrm{d}\gamma}  &
\leq \frac{\delta_{\mathrm{c}}P(\gamma)}{n\gamma},\  \forall \gamma>0,
\label{pleP5}%
\end{align}
where $\delta_{\mathrm{c}}$ is a constant independent of $\gamma$. Finally,
let $\gamma_{0}$ be a given constant and $\delta \triangleq2\alpha( A) -1\geq
1$. Then, there holds
\begin{equation}
\mu_{1}\gamma^{\delta}I_{n}\leq P(\gamma)\leq \mu_{2}\gamma I_{n},\  \gamma
\in(0,\gamma_{0}], \label{pleP6}%
\end{equation}
where $\mu_{1}=\mu_{1}(A,B) $ is a constant depending on $(A,B) $ and $\mu
_{2}=\mu_{2}(\gamma_{0}) $ is a constant depending on $\gamma_{0}$.
\end{lemma}

Equation (\ref{pleP1}), inequalities (\ref{pleP2}), (\ref{pleP3}),
(\ref{pleP6}), and the first inequality in (\ref{pleP5}) are referenced from
\cite{zhou14book} and \cite{zhou20auto}. The proof for (\ref{pleP4}) will be
provided in Appendix A2. The second inequality in (\ref{pleP5}) can be viewed
as a special case of (\ref{plePP3}) and its proof is omitted.


\section{State Feedback Stabilization}

To proceed, we impose the following assumption instead of Assumption
\ref{ass2}.

\begin{assumption}
\label{ass3}There exists two unknown constants $\bar{\tau}>0$ and $d\in
\lbrack0,1)$ such that%
\[
0\leq \tau_{i}(t)\leq \bar{\tau},\quad \dot{\tau}_{i}(t)\leq d,\quad \forall
t\geq0,\quad i\in \mathbf{I}_{1}^{q}.
\]

\end{assumption}

We can present the key finding in this section.

\begin{theorem}
\label{thm1} Let Assumptions \ref{ass1} and \ref{ass3} be satisfied. Then
system (\ref{sys}) is asymptotically stabilized by the controller%
\begin{align}
u(t)  &  =-{B^{\mathrm{T}}P(\gamma(t))}x(t) ,\label{eq3.3}\\
\gamma(t)  &  =\frac{\gamma_{0}}{(\omega t+1)^{\mu}},\quad t\in[ 0,\infty) ,
\label{eqgammat}%
\end{align}
where $P(\gamma)$ is the unique positive definite solution to the PLE
(\ref{ple}), and $\gamma_{0}>0,\omega>0,\mu \in(0,1)$ are constants.
\end{theorem}

\begin{proof}
It can be verified that $\gamma(t)$ satisfies the following differential
equation
\begin{equation}
\dot{\gamma}(t)=-k\gamma^{1+1/\mu}(t),\  \gamma(0)=\gamma_{0},\ t\in[ 0,\infty)
, \label{eq3.2}%
\end{equation}
where $k=\omega \mu/\gamma_{0}^{1/\mu}>0$ is a constant. It follows that
$\dot{\gamma}(t)<0$ and $\gamma(t)\leq \gamma_{0},\forall t\geq0$. Going
forward, unless explicitly stated, we will omit the explicit dependence of
variables on $t$. For instance, $x$ represents $x(t)$. Furthermore, we denote
$P(t)=P(\gamma(t)) $.

Denote $\phi_{i}(t)=t-\tau_{i}(t),i\in \mathbf{I}_{1}^{q}$ and%
\begin{equation}
\mathit{\Delta}(t)=\sum \limits_{i=1}^{q}\mathit{\Delta}_{i}%
(t),\  \mathit{\Delta}_{i}(t)=B_{i}B^{\mathrm{T}}(P(t)x(t)-P(\phi_{i}%
)x(\phi_{i})). \label{eq3.5}%
\end{equation}
Then the closed-loop system consisting of (\ref{sys}) and (\ref{eq3.3}) can be
written as%
\begin{align}
\dot{x}(t)=  &  Ax(t)-\sum \limits_{i=1}^{q}B_{i}B^{\mathrm{T}}P(\phi
_{i})x(\phi_{i})\nonumber \\
=  &  (A-BB^{\mathrm{T}}P(t))x(t)+BB^{\mathrm{T}}P(t)x(t)-\sum \limits_{i=1}%
^{q}B_{i}B^{\mathrm{T}}P(\phi_{i})x(\phi_{i})\nonumber \\
=  &  (A-BB^{\mathrm{T}}P(t))x(t)+\sum \limits_{i=1}^{q}\mathit{\Delta}%
_{i}(t)\nonumber \\
=  &  (A-BB^{\mathrm{T}}P(t))x(t)+\mathit{\Delta}(t),\  \forall t\geq \bar{\tau
}. \label{eq3.6}%
\end{align}

Consider the following Lyapunov-like function%
\begin{equation}
V_{0}(t)=V_{0}(t,x(t))=x^{\mathrm{T}}(t)P(t)x(t), \label{eq3.7}%
\end{equation}
whose time-derivative along the trajectories of system (\ref{eq3.6}) can be
evaluated as
\begin{align}
\dot{V}_{0}(t)=  &  \dot{x}^{\mathrm{T}}P(t)x+x^{\mathrm{T}}\dot
{P}(t)x+x^{\mathrm{T}}P(t)\dot{x}\nonumber \\
=  &  x^{\mathrm{T}}(A^{\mathrm{T}}P(t)+P(t)A-2P(t)BB^{\mathrm{T}}%
P(t))x+\dot{\gamma}x^{\mathrm{T}}\frac{\partial P(t)}{\partial \gamma
}x+2\mathit{\Delta}^{\mathrm{T}}P(t)x\nonumber \\
\leq &  -\gamma x^{\mathrm{T}}P(t)x+\dot{\gamma}x^{\mathrm{T}}\frac{\partial
P(t)}{\partial \gamma}x+2\mathit{\Delta}^{\mathrm{T}}P(t)x\nonumber \\
\leq &  -\gamma x^{\mathrm{T}}P(t)x+\frac{\gamma}{2}x^{\mathrm{T}}%
P(t)x+\frac{2}{\gamma}\mathit{\Delta}^{\mathrm{T}}P(t)\mathit{\Delta
}\nonumber \\
\leq &  -\frac{\gamma}{2}x^{\mathrm{T}}P(t)x+\frac{2}{\gamma}\mathit{\Delta
}^{\mathrm{T}}P(t)\mathit{\Delta}, \label{eq3.8}%
\end{align}
where the first inequality holds due to (\ref{ple}), the second inequality
holds due to the Young's inequality and (\ref{pleP5}).

Let $\psi(t)=P(t)x(t)$. Then we have%
\begin{align*}
\dot{\psi}(t)  &  =P(t)\dot{x}+\dot{P}(t)x\\
&  =P(t)\left(  Ax-\sum \limits_{i=1}^{q}B_{i}B^{\mathrm{T}}P(\phi_{i}%
)x(\phi_{i})\right)  +\dot{\gamma}\frac{\partial P(t)}{\partial \gamma}x\\
&  =\left(  P(t)A+\dot{\gamma}\frac{\partial P(t)}{\partial \gamma}\right)
x-P(t)\sum \limits_{j=1}^{q}B_{j}B^{\mathrm{T}}P(\phi_{j})x(\phi_{j}).
\end{align*}
Integrating both sides of the above equation from $t-\tau_{i}(t)$ to $t$ and
then multiplying both sides from the left by $B_{i}B^{\mathrm{T}}$ gives, for
all $i\in \mathbf{I}_{1}^{q},$%
\begin{align}
\mathit{\Delta}_{i}(t)=  &  B_{i}B^{\mathrm{T}}(\psi(t)-\psi(\phi
_{i}))
=     B_{i}B^{\mathrm{T}}\int_{t-\tau_{i}}^{t}(\xi_{1}(s)-P(s)\xi
_{2}(s))\mathrm{d}s,\  \forall t\geq2\bar{\tau}, \label{eq3.9}%
\end{align}
where
\[
\xi_{1}(t)=\left(  P(t)A+\dot{\gamma}(t)\frac{\partial P(t)}{\partial \gamma
}\right)  x,\ \xi_{2}(t)=\sum \limits_{j=1}^{q}B_{j}B^{\mathrm{T}}P(\phi
_{j})x(\phi_{j}).
\]
By using the above relation (\ref{eq3.9}), together with (\ref{eq3.5}), we
deduce that
\begin{align}
\mathit{\Delta}^{\mathrm{T}}P(t)\mathit{\Delta}\leq &  \mu_{2}\gamma \left(
\sum \limits_{i=1}^{q}\mathit{\Delta}_{i}\right)  ^{\mathrm{T}}\left(
\sum \limits_{i=1}^{q}\mathit{\Delta}_{i}\right) \nonumber \\
\leq &  q\mu_{2}\gamma \sum \limits_{i=1}^{q}\mathit{\Delta}_{i}^{\mathrm{T}%
}\mathit{\Delta}_{i}\nonumber \\
\leq &  q\mu_{2}\gamma \sum \limits_{i=1}^{q}\left(  (1+k_{1})\beta_{1i}+\left(
1+\frac{1}{k_{1}}\right)  \beta_{2i}\right)  , \label{eq3.10}%
\end{align}
where the first inequality follows from (\ref{pleP6}), the second inequality
follows from the Jensen's inequality, and the final inequality is a
consequence of the Young's inequality with $k_{1}>0$ being a constant to be
determined, and%
\begin{align*}
\beta_{1i}  &  =\left(  \int_{t-\tau_{i}}^{t}\xi_{1}(s)\mathrm{d}s\right)
^{\mathrm{T}}BB_{i}^{\mathrm{T}}B_{i}B^{\mathrm{T}}\left(  \int_{t-\tau_{i}%
}^{t}\xi_{1}(s)\mathrm{d}s\right)  ,\\
\beta_{2i}  &  =\left(  \int_{t-\tau_{i}}^{t}P(s)\xi_{2}(s)\mathrm{d}s\right)
^{\mathrm{T}}BB_{i}^{\mathrm{T}}B_{i}B^{\mathrm{T}}\left(  \int_{t-\tau_{i}%
}^{t}P(s)\xi_{2}(s)\mathrm{d}s\right)  .
\end{align*}
By using the Jensen's inequality, the Young's inequality, and the definition
\[
\mathcal{P}(s)=\left(  \frac{\partial P(s)}{\partial \gamma}\right)
^{\mathrm{T}}BB^{\mathrm{T}}\left(  \frac{\partial P(s)}{\partial \gamma
}\right)  ,
\]
it is easy to get that
\begin{align}
\beta_{1i}\leq &  \tau_{i}\int_{t-\tau_{i}}^{t}\xi_{1}^{\mathrm{T}}%
(s)BB_{i}^{\mathrm{T}}B_{i}B^{\mathrm{T}}\xi_{1}(s)\mathrm{d}s\nonumber \\
\leq &  \bar{\tau}\left \Vert B_{i}\right \Vert ^{2}\int_{t-\bar{\tau}}^{t}%
\xi_{1}^{\mathrm{T}}(s)BB^{\mathrm{T}}\xi_{1}(s)\mathrm{d}s\nonumber \\
\leq &  2\bar{\tau}\left \Vert B_{i}\right \Vert ^{2}\int_{t-\bar{\tau}}%
^{t}x^{\mathrm{T}}(s)A^{\mathrm{T}}P(s)BB^{\mathrm{T}}P(s)Ax(s)\mathrm{d}%
s+2\bar{\tau}\left \Vert B_{i}\right \Vert ^{2}\int_{t-\bar{\tau}}^{t}%
\dot{\gamma}^{2}(s)x^{\mathrm{T}}(s)\mathcal{P}(s)x(s)\mathrm{d}s\nonumber \\
\leq &  2\bar{\tau}n\left \Vert B_{i}\right \Vert ^{2}\int_{t-\bar{\tau}}%
^{t}\gamma(s)x^{\mathrm{T}}(s)A^{\mathrm{T}}P(s)Ax(s)\mathrm{d}s+2\bar{\tau
}n\left \Vert B_{i}\right \Vert ^{2}\int_{t-\bar{\tau}}^{t}\dot{\gamma}%
^{2}(s)x^{\mathrm{T}}(s)\frac{\partial P(s)}{\partial \gamma}x(s)\mathrm{d}%
s\nonumber \\
\leq &  6\bar{\tau}n^{3}\left \Vert B_{i}\right \Vert ^{2}\int_{t-\bar{\tau}%
}^{t}\gamma^{3}(s)x^{\mathrm{T}}(s)P(s)x(s)\mathrm{d}s+2\bar{\tau}%
\delta_{\mathrm{c}}\left \Vert B_{i}\right \Vert ^{2}\int_{t-\bar{\tau}}%
^{t}\frac{\dot{\gamma}^{2}(s)}{\gamma(s)}x^{\mathrm{T}}(s)P(s)x(s)\mathrm{d}s,
\label{eqbeita1}%
\end{align}
\qquad \qquad where we have used inequalities (\ref{pleP2}), (\ref{pleP3}),
(\ref{pleP4}) and (\ref{pleP5}) and the fact that $\tau_{i}\leq \bar{\tau}$. It
follows from (\ref{eq3.2}) that, for all $t\geq0$,%
\begin{align*}
\dot{\gamma}^{2}(t)=  &  k^{2}\gamma^{2+2/\mu}(t)
=    k^{2}\gamma^{4}(t)\gamma^{2/\mu-2}(t)
\leq   k^{2}\gamma_{0}^{2/\mu-2}\gamma^{4}(t),
\end{align*}
where we have noticed that $2/\mu-2>0$. Then the inequality in (\ref{eqbeita1}%
) can be further written as, for all $i\in \mathbf{I}_{1}^{q},$
\begin{align}
\beta_{1i}\leq &  6\bar{\tau}n^{3}\left \Vert B_{i}\right \Vert ^{2}\int
_{t-\bar{\tau}}^{t}\gamma^{3}(s)x^{\mathrm{T}}(s)P(s)x(s)\mathrm{d}%
s+2\bar{\tau}\delta_{\mathrm{c}}\left \Vert B_{i}\right \Vert ^{2}\int
_{t-\bar{\tau}}^{t}k^{2}\gamma_{0}^{2/\mu-2}\gamma^{3}(s)x^{\mathrm{T}%
}(s)P(s)x(s)\mathrm{d}s\nonumber \\
=  &  2\bar{\tau}\left \Vert B_{i}\right \Vert ^{2}(3n^{3}+\delta_{\mathrm{c}%
}k^{2}\gamma_{0}^{2/\mu-2})\int_{t-\bar{\tau}}^{t}\gamma^{3}(s)V_{0}%
(s)\mathrm{d}s\nonumber \\
\leq &  2\bar{\tau}\left \Vert B_{i}\right \Vert ^{2}(3n^{3}+\delta_{\mathrm{c}%
}k^{2}\gamma_{0}^{2/\mu-2})\int_{t-2\bar{\tau}}^{t}\gamma^{3}(s)V_{0}%
(s)\mathrm{d}s. \label{eq3.12}%
\end{align}

Similarly, by using the Jensen's inequality, the fact that $\tau_{i}\leq
\bar{\tau}$, and inequalities (\ref{pleP2}) and (\ref{pleP6}), we have%
\begin{align}
\beta_{2i}\leq &  \left \Vert B_{i}\right \Vert ^{2}\tau_{i}\int_{t-\tau_{i}%
}^{t}\xi_{2}^{\mathrm{T}}(s)P(s)BB^{\mathrm{T}}P(s)\xi_{2}(s)\mathrm{d}%
s\nonumber \\
\leq &  \bar{\tau}n\left \Vert B_{i}\right \Vert ^{2}\int_{t-\bar{\tau}}%
^{t}\gamma(s)\xi_{2}^{\mathrm{T}}(s)P(s)\xi_{2}(s)\mathrm{d}s\nonumber \\
\leq &  \bar{\tau}n\mu_{2}\left \Vert B_{i}\right \Vert ^{2}\int_{t-\bar{\tau}%
}^{t}\gamma^{2}(s)\left \Vert \xi_{2}(s)\right \Vert ^{2}\mathrm{d}s.
\label{eqbeita2}%
\end{align}
By using the Jensen's inequality and (\ref{pleP2}) we have%
\begin{align*}
\left \Vert \xi_{2}(s)\right \Vert ^{2}  &  \leq q\sum \limits_{j=1}%
^{q}x^{\mathrm{T}}(\phi_{j}(s))P(\phi_{j}(s))BB_{j}^{\mathrm{T}}%
B_{j}B^{\mathrm{T}}P(\phi_{j}(s))x(\phi_{j}(s))\\
&  \leq q\sum \limits_{j=1}^{q}\left \Vert B_{j}\right \Vert ^{2}x^{\mathrm{T}%
}(\phi_{j}(s))P(\phi_{j}(s))BB^{\mathrm{T}}P(\phi_{j}(s))x(\phi_{j}(s))\\
&  \leq nq\sum \limits_{j=1}^{q}\left \Vert B_{j}\right \Vert ^{2}\gamma(\phi
_{j}(s))x^{\mathrm{T}}(\phi_{j}(s))P(\phi_{j}(s))x(\phi_{j}(s))\\
&  =nq\sum \limits_{j=1}^{q}\left \Vert B_{j}\right \Vert ^{2}\gamma(\phi
_{j}(s))V_{0}(\phi_{j}(s)),
\end{align*}
substitution of which into (\ref{eqbeita2}) gives
\begin{align}
\beta_{2i}\leq &  \bar{\tau}n^{2}\mu_{2}q\left \Vert B_{i}\right \Vert ^{2}%
\int_{t-\bar{\tau}}^{t}\sum \limits_{j=1}^{q}\left \Vert B_{j}\right \Vert
^{2}\gamma^{2}(s)\gamma(\phi_{j}(s))V_{0}(\phi_{j}(s))\mathrm{d}s\nonumber \\
\leq &  \bar{\tau}n^{2}\mu_{2}q\left \Vert B_{i}\right \Vert ^{2}\int
_{t-\bar{\tau}}^{t}\sum \limits_{j=1}^{q}\left \Vert B_{j}\right \Vert ^{2}%
\gamma^{3}(\phi_{j}(s))V_{0}(\phi_{j}(s))\mathrm{d}s\nonumber \\
\leq &  \bar{\tau}n^{2}\mu_{2}q\left \Vert B_{i}\right \Vert ^{2}\sum
\limits_{j=1}^{q}\left \Vert B_{j}\right \Vert ^{2}\sum \limits_{j=1}^{q}%
\int_{t-\bar{\tau}}^{t}\gamma^{3}(\phi_{j}(s))V_{0}(\phi_{j}(s))\mathrm{d}s,
\label{eq3.14r1}%
\end{align}
where we have used the fact that $\gamma(t)\leq \gamma(\phi_{j}(t))$. By using
the denotation $\theta_{j}=s-\tau_{j}(s)=\phi_{j}(s)$ and the fact
$\mathrm{d}\theta_{j}=(1-\dot{\tau}_{j}(s))\mathrm{d}s$, we have%
\begin{align*}
\sum \limits_{j=1}^{q}\int_{t-\bar{\tau}}^{t}\gamma^{3}(\phi_{j}(s))V_{0}%
(\phi_{j}(s))\mathrm{d}s  &  =\sum \limits_{j=1}^{q}\frac{1}{1-\dot{\tau}_{j}%
}\int_{t-\bar{\tau}-\tau_{j}(t-\bar{\tau})}^{t-\tau_{j}(t)}\gamma^{3}%
(\theta_{j})V_{0}(\theta_{j})\mathrm{d}\theta_{j}\\
&  \leq \frac{1}{1-d}\sum \limits_{j=1}^{q}\int_{t-2\bar{\tau}}^{t}\gamma
^{3}(\theta_{j})V_{0}(\theta_{j})\mathrm{d}\theta_{j}\\
&  =\frac{q}{1-d}\int_{t-2\bar{\tau}}^{t}\gamma^{3}(s)V_{0}(s)\mathrm{d}s,
\end{align*}
where we have used the condition $\dot{\tau}_{i}(t)\leq d<1,\ i\in
\mathbf{I}_{1}^{q}$, as stated in Assumption \ref{ass3}. With this, the
inequality in (\ref{eq3.14r1}) can be further written as
\begin{equation}
\beta_{2i}\leq \frac{\bar{\tau}n^{2}\mu_{2}q^{2}}{1-d}\left \Vert B_{i}%
\right \Vert ^{2}\sum \limits_{j=1}^{q}\left \Vert B_{j}\right \Vert ^{2}%
\int_{t-2\bar{\tau}}^{t}\gamma^{3}(s)V_{0}(s)\mathrm{d}s. \label{eq3.15}%
\end{equation}

It is trivial to show that there is a $k_{1}>0$ such that
\begin{align*}
\frac{k_{2}}{2}\triangleq &  (1+k_{1})2\bar{\tau}\sum \limits_{i=1}%
^{q}\left \Vert B_{i}\right \Vert ^{2}(3n^{3}+\delta_{\mathrm{c}}k^{2}\gamma
_{0}^{2/\mu-2})
=     \left(  1+\frac{1}{k_{1}}\right)  \frac{\bar{\tau}n^{2}\mu_{2}q^{2}%
}{1-d}\left(  \sum \limits_{i=1}^{q}\left \Vert B_{i}\right \Vert ^{2}\right)
^{2}.
\end{align*}
Then it follows from (\ref{eq3.10}), (\ref{eq3.12}) and (\ref{eq3.15}) that%
\begin{align}
\mathit{\Delta}^{\mathrm{T}}P(t)\mathit{\Delta}\leq &  q\mu_{2}\gamma \left(
(1+k_{1})\sum \limits_{i=1}^{q}\beta_{1i}(t)+\left(  1+\frac{1}{k_{1}}\right)
\sum \limits_{i=1}^{q}\beta_{2i}(t)\right) \nonumber \\
\leq &  q\mu_{2}\gamma \left(  \frac{k_{2}}{2}\int_{t-2\bar{\tau}}^{t}%
\gamma^{3}(s)V_{0}(s)\mathrm{d}s+\frac{k_{2}}{2}\int_{t-2\bar{\tau}}^{t}%
\gamma^{3}(s)V_{0}(s)\mathrm{d}s\right) \nonumber \\
=  &  k_{2}q\mu_{2}\gamma \int_{t-2\bar{\tau}}^{t}\gamma^{3}(s)V_{0}%
(s)\mathrm{d}s. \label{eq3.16}%
\end{align}
Subsequently, substituting (\ref{eq3.16}) into (\ref{eq3.8}) gives%
\begin{equation}
\dot{V}_{0}(t)\leq-\frac{\gamma}{2}V_{0}(t)+2k_{2}q\mu_{2}\int_{t-2\bar{\tau}%
}^{t}\gamma^{3}(s)V_{0}(s)\mathrm{d}s. \label{eq3.17}%
\end{equation}
Let the restriction of the solution $x(t)$ of system (\ref{eq3.6}) to the
interval $\left[  t-2\bar{\tau},t\right]  $ be denoted by $x_{t}:\theta \mapsto
x(t+\theta),\theta \in \left[  -2\bar{\tau},0\right]  $, and $k_{3}\geq
4k_{2}q\mu_{2}\bar{\tau}$ be a constant. Choose the following
Lyapunov-Krasovskii-like functional
\begin{equation}
V(t,x_{t})=V_{0}(t)+k_{3}\int_{t-2\bar{\tau}}^{t}\left(  3+\frac{s-t}%
{\bar{\tau}}\right)  \gamma^{3}(s)V_{0}(s)\mathrm{d}s, \label{eq3.18}%
\end{equation}
whose time-derivative along the trajectories of system (\ref{eq3.6}) can be
evaluated as
\begin{align}
\dot{V}(t,x_{t})=  &  \dot{V}_{0}(t)+3k_{3}\gamma^{3}(t)V_{0}(t)-k_{3}%
\gamma^{3}(t-2\bar{\tau})V_{0}(t-2\bar{\tau})-\frac{k_{3}}{\bar{\tau}}%
\int_{t-2\bar{\tau}}^{t}\gamma^{3}(s)V_{0}(s)\mathrm{d}s\nonumber \\
\leq &  -\frac{\gamma}{2}V_{0}(t)+2k_{2}q\mu_{2}\int_{t-2\bar{\tau}}^{t}%
\gamma^{3}(s)V_{0}(s)\mathrm{d}s+3k_{3}\gamma^{3}(t)V_{0}(t)\nonumber \\
&  -k_{3}\gamma^{3}(t-2\bar{\tau})V_{0}(t-2\bar{\tau})-\frac{k_{3}}{\bar{\tau
}}\int_{t-2\bar{\tau}}^{t}\gamma^{3}(s)V_{0}(s)\mathrm{d}s\nonumber \\
\leq &  \left(  -\frac{\gamma}{2}+3k_{3}\gamma^{3}\right)  V_{0}(t)+\left(
2k_{2}q\mu_{2}-\frac{k_{3}}{\bar{\tau}}\right)  \int_{t-2\bar{\tau}}^{t}%
\gamma^{3}(s)V_{0}(s)\mathrm{d}s\nonumber \\
\leq &  \left(  -\frac{\gamma}{2}+3k_{3}\gamma^{3}\right)  V_{0}%
(t)-\frac{k_{3}}{2\bar{\tau}}\int_{t-2\bar{\tau}}^{t}\gamma^{3}(s)V_{0}%
(s)\mathrm{d}s, \label{eq3.19}%
\end{align}
where we have used (\ref{eq3.2}) and (\ref{eq3.17}). Clearly, there exists a
constant $t_{1}\in \lbrack2\bar{\tau},\infty)$ independent of the initial
condition such that
\[
\gamma(t)\leq \min \left \{  \frac{1}{\sqrt{12k_{3}}},\frac{2}{3\bar{\tau}%
}\right \}  ,
\]
is satisfied for all $t\geq t_{1},$ namely, $3k_{3}\gamma^{3}\leq \gamma/4$ and
$-1/(2\bar{\tau})\leq-3\gamma/4$ are satisfied for all $t\geq t_{1}.$ Then it
follows from (\ref{eq3.19}) that
\begin{align*}
\dot{V}(t,x_{t})  &  \leq-\frac{\gamma}{4}V_{0}(t)-\frac{3\gamma}{4}k_{3}%
\int_{t-2\bar{\tau}}^{t}\gamma^{3}(s)V_{0}(s)\mathrm{d}s\\
&  \leq-\frac{\gamma}{4}V_{0}(t)-\frac{\gamma}{4}k_{3}\int_{t-2\bar{\tau}}%
^{t}\left(  3+\frac{s-t}{\bar{\tau}}\right)  \gamma^{3}(s)V_{0}(s)\mathrm{d}%
s\\
&  =-\frac{\gamma}{4}V(t,x_{t}),\  \forall t\geq t_{1}.
\end{align*}
By using the comparison principle, we have
\begin{align}
V(t,x_{t})  &  \leq \exp \left(  -\frac{1}{4}\int_{t_{1}}^{t}\gamma
(s)\mathrm{d}s\right)  V(t_{1},x_{t_{1}})\nonumber \\
&  =\exp \left(  -\frac{\gamma_{0}}{4}\int_{t_{1}}^{t}\frac{\mathrm{d}%
s}{(\omega s+1)^{\mu}}\right)  V(t_{1},x_{t_{1}})\nonumber \\
&  =\exp \left(  -\frac{\gamma_{0}}{4}\int_{t_{1}}^{t}\frac{\mathrm{d}(\omega
s+1)^{1-\mu}}{\omega(1-\mu)}\right)  V(t_{1},x_{t_{1}})\nonumber \\
&  =\exp \left(  -\frac{\gamma_{0}((\omega t+1)^{1-\mu}-(\omega t_{1}%
+1)^{1-\mu})}{4\omega(1-\mu)}\right)  V(t_{1},x_{t_{1}})\nonumber \\
&  =\frac{f(t_{1})}{f(t)}V(t_{1},x_{t_{1}}),\  \forall t\geq t_{1},
\label{eq3.24}%
\end{align}
where
\begin{equation}
f(t)=\exp \left(  \frac{\gamma_{0}(\omega t+1)^{1-\mu}}{4\omega(1-\mu)}\right)
. \label{eq3.add24}%
\end{equation}
In addition, it follows from (\ref{eq3.7}) and (\ref{eq3.18}) that
\begin{align}
V(t,x_{t})  &  \geq V_{0}(t)=x^{\mathrm{T}}(t)P(t)x(t)
   \geq \mu_{1}\gamma^{\delta}\left \Vert x(t)\right \Vert ^{2}
  =\frac{\mu_{1}\gamma_{0}^{\delta}}{(\omega t+1)^{\delta \mu}}\left \Vert
x(t)\right \Vert ^{2}, \label{eq3.25}%
\end{align}
where the second inequality holds due to (\ref{pleP6}). By using
(\ref{eq3.24}) and (\ref{eq3.25}), we deduce that%
\begin{equation}
\left \Vert x(t)\right \Vert ^{2}\leq \frac{f(t_{1})}{\mu_{1}\gamma_{0}^{\delta}%
}\  \frac{(\omega t+1)^{\delta \mu}}{f(t)}V(t_{1},x_{t_{1}}),\  \forall t\geq
t_{1}. \label{eq3.26}%
\end{equation}

In addition, as the closed-loop system is a linear time-delay system with
uniformly bounded coefficients, for any $T\geq0$, there exists a constant
$c_{1}=c_{1}(T)>0$ such that (see a proof in Appendix A3)%
\begin{equation}
\left \Vert x(t)\right \Vert \leq c_{1}\left(  \left \Vert x_{0}\right \Vert +
\sup_{\theta \in[ -\bar{\tau},0)}\left \Vert u_{0}(\theta)\right \Vert \right)
,\forall t\in[0,T]. \label{eqadd1}%
\end{equation}
Thus, it follows from $V(t,x_{t})\leq c_{2}\left \Vert x_{t}\right \Vert ^{2},$
where $c_{2}>0$ is a constant, and%
\[
\lim_{t\rightarrow \infty}\frac{(\omega t+1) ^{\delta \mu}}{f(t) }%
=\lim_{s\rightarrow \infty}\frac{s^{\frac{\delta \mu}{1-\mu}}}{\exp \left(
\frac{\gamma_{0}}{4\omega(1-\mu)}s\right)  }=0,
\]
that {$\lim_{t\rightarrow \infty}\left \Vert x(t)\right \Vert =0,$ namely, the
zero solution of the closed-loop system (\ref{eq3.6}) is attractive}. Also, it
follows from (\ref{eq3.26}) and (\ref{eqadd1}) that, for any given
$\varepsilon>0,$ there exists a $\delta_{0}>0$ such that $(\Vert x_{0} \Vert+
\sup_{\theta \in[ -\bar{\tau},0)}\Vert u_{0}(\theta)\Vert) \leq \delta
_{0}\Rightarrow \left \Vert x(t)\right \Vert \leq \varepsilon,\forall t\geq0,$
namely, the zero solution of the closed-loop system (\ref{eq3.6}) is also
stable in the sense of Lyapunov. Thus (the zero solution of) the closed-loop
system (\ref{eq3.6}) is asymptotically stable, and the proof is finished.
\end{proof}

\begin{remark}
We provide some further explanations on the necessity of Assumption
\ref{ass1}. If $A$ has eigenvalues in the right half-plane, a memory-based
controller becomes necessary, as it must account for the time delay $\tau_{i}$
\cite{zhou14book}. However, in this paper, we assume that the time delay is
unknown. If $A$ has nonzero eigenvalues on the imaginary axis, though the
controller is allowed to be memoryless, the feedback gain should be dependent
on the delays $\tau_{i}$ \cite{zhou14book}, which, however, are assumed to be
unknown in this paper. It follows that, under the assumption that the delays
(and their bounds) are unknown, linear systems with only zero eigenvalues
might be the largest class of systems that can be stabilized by memoryless
state feedback.
\end{remark}

\begin{remark}
In Theorem \ref{thm1}, we utilize the assumption $\dot{\tau}_{{i}}(t)\leq
d<1$, which arises from employing the Lyapunov-Krasovskii functional
\cite{Fridman2014book}. Typically, the approach based on the Razumikhin
theorem does not require this assumption. However, it is not applicable to the
scenario in this paper. The reason is that the function $V_{{0}}(t)$ in
(\ref{eq3.7}) deviates from the conventional definition of a Lyapunov function
in the sense that it lacks a lower bound. Specifically, there is no constant
$d_{1}>0$ such that $V_{0}(t)\geq d_{1}\left \Vert x\right \Vert ^{2}$.
\end{remark}

\begin{remark}
To tackle the challenge of unknown time-varying delays, we introduce the
time-varying parameter $\gamma(t)$ as described in (\ref{eqgammat}). As a
result, the closed-loop system can only achieve asymptotic stability rather
than exponential stability. To illustrate this, let us consider a scalar
delay-free system%
\begin{equation}
\dot{x}(t)=u(t), \label{eq3.27}%
\end{equation}
which satisfies Assumptions \ref{ass1} and \ref{ass3}. According to Theorem
\ref{thm1}, we design a relevant controller%
\begin{equation}
u(t)=-\gamma(t)x(t)=-\frac{\gamma_{0}}{(\omega t+1)^{\mu}}x(t). \label{eq3.28}%
\end{equation}
Then, it follows from (\ref{eq3.27}) and (\ref{eq3.28}) that
\begin{equation}
\dot{x}(t)=-\frac{\gamma_{0}}{(\omega t+1)^{\mu}}x(t). \label{eq3.29}%
\end{equation}
Solving the differential equation in (\ref{eq3.29}) gives
\[
x(t)=\exp \left(  -\int_{0}^{t}\frac{\gamma_{0}}{(\omega s+1)^{\mu}}%
\mathrm{d}s\right)  x(0)=\frac{\exp \left(  \frac{\gamma_{0}}{\omega(1-\mu
)}\right)  }{\exp \left(  \frac{\gamma_{0}(\omega t+1)^{1-\mu}}{\omega(1-\mu
)}\right)  }x(0),
\]
from which, it is evident that the closed-loop system (\ref{eq3.29}) achieves
only asymptotic convergence instead of exponential convergence. This is due to
the fact that, for any constants $\theta_{1}>0$ and $\theta_{2}>0$, there
always exists a finite time $t_{\ast}>0$ such that
\[
\frac{\exp \left(  \frac{\gamma_{0}}{\omega(1-\mu)}\right)  }{\exp \left(
\frac{\gamma_{0}(\omega t+1)^{1-\mu}}{\omega(1-\mu)}\right)  }>\theta_{2}%
\exp(-\theta_{1}t),\  \forall t>t_{\ast}.
\]

\end{remark}

\section{Observer-Based Output Feedback Stabilization}

In this paper, we assume that $\tau_{i}(t),i\in \mathbf{I}_{{1}}^{q}$ are
unknown in Assumption \ref{ass3}, which implies that $\mathit{\Sigma}%
_{i=1}^{q}B_{i}u(t-\tau_{i}(t))$ in (\ref{eq4.0}) is unavailable. Hence the
observer-based controller (\ref{eq4.0}) is not implementable. It follows from
(\ref{eq4.0}) that
\begin{align*}
\dot{\xi}(t)=  &  A\xi(t)+\sum \limits_{i=1}^{q}B_{i}u(t-\tau_{i}%
(t))+L(y(t)-C\xi(t))\\
=  &  (A-LC)\xi(t)+Ly(t)-\sum \limits_{i=1}^{q}B_{i}{B^{\mathrm{T}}%
P(\gamma(t-\tau_{i}(t)))}\xi(t-\tau_{i}(t)).
\end{align*}
Let $A-LC$ be Hurwitz. As $\lim_{t\rightarrow \infty}{P(\gamma(t))=0}$ and
$y(t)$ is independent of $\xi(t)$, the last term $-\mathit{\Sigma}_{i=1}%
^{q}B_{i}B^{\mathrm{T}}$ $P(\gamma(t-\tau_{i}(t)))\xi(t-\tau_{i}(t))$ is
dominated by the first term $(A-LC)\xi(t)$. As a result, the observer in
(\ref{eq4.0}) can be truncated as
\[
\dot{\xi}(t)=A\xi(t)+L(y(t)-C\xi(t)),
\]
as $t\rightarrow \infty$. This observation leads us to propose the following
alternative observer-based controller
\begin{equation}
\left \{
\begin{array}
[c]{l}%
\dot{\xi}(t)=A\xi(t)+L(y(t)-C\xi(t)),\\
u(t)=-{B^{\mathrm{T}}P(\gamma(t))}\xi(t),\quad t\geq0,
\end{array}
\right.  \label{eq4.1}%
\end{equation}
where $\xi \in \mathbf{R}^{n}$ is the observed state, $L$ is the observer gain,
$P(\gamma)$ is the solution to the PLE (\ref{ple}) with $\gamma=\gamma(t)$
defined in (\ref{eqgammat}). Of course, with the above truncation, the
observer error $e=\xi-x$ no longer satisfies
\[
\dot{e}(t)=(A-LC)e(t),
\]
and thus the stability of the overall system should be carefully analyzed.

\begin{theorem}
\label{thm2} Let Assumptions \ref{ass1} and \ref{ass3} be satisfied, $(A,C)$
be observable, and $L$ be chosen such that $A-LC$ is Hurwitz. Then the
observer-based controller (\ref{eq4.1}) with $\gamma(t)$ defined in
(\ref{eqgammat}) asymptotically stabilizes system (\ref{sys}).
\end{theorem}

\begin{proof}
Following a similar approach in the proof of Theorem \ref{thm1}, we consider
the closed-loop system with $t \geq2\bar{\tau}$. Let $e(t)=\xi(t)-x(t)$ and
denote $P(t)=P(\gamma(t)) $. We can express the system consisting of
(\ref{sys}) and (\ref{eq4.1}) as
\begin{align}
\dot{e}(t)  &  = (A-LC) e(t)+ \sum \limits_{i=1}^{q}B_{i}{B^{\mathrm{T}}%
P(\phi_{i})}(x(\phi_{i}) +e(\phi_{i})) ,\label{eq4.2add}\\
\dot{x}(t)  &  =Ax(t)-\sum \limits_{i=1}^{q}B_{i}{B^{\mathrm{T}}P(\phi_{i}%
)}(x(\phi_{i}) +e(\phi_{i})) . \label{eq4.2}%
\end{align}
Denote
\[
\mathit{\Xi}(t)=\mathit{\Delta}-\sum \limits_{i=1}^{q}B_{i}B^{\mathrm{T}%
}P{(\phi_{i})}e(\phi_{i}) ,
\]
where $\mathit{\Delta}$ is defined in (\ref{eq3.5}). Then (\ref{eq4.2}) can be
further written as%
\begin{equation}
\dot{x}(t)=(A-BB^{\mathrm{T}}P(t)) x(t)+\mathit{\Xi}(t). \label{eq4.4}%
\end{equation}
In the remaining of the proof we will omit the explicit dependence of
variables on $t$ when necessary.

As $A-LC$ is Hurwitz, there exists a constant $\rho>0$ and a matrix $Q>0$ such
that
\begin{equation}
(A-LC)^{\mathrm{T}}Q+Q(A-LC)\leq-\rho Q. \label{eq4.5}%
\end{equation}
We construct the function%
\[
U_{0}(t)=U_{0}(e(t))=\mathcal{\kappa}_{0}e^{\mathrm{T}}(t)Qe(t),
\]
where $\mathcal{\kappa}_{0}>0$ is a constant to be designed. Defining%
\begin{equation}
\xi_{3}(t)=\sum \limits_{i=1}^{q}B_{i}{B^{\mathrm{T}}P(\phi_{i})}(x(\phi
_{i})+e(\phi_{i})), \label{eq4.add6}%
\end{equation}
and differentiating $U_{0}(t)$ along the trajectories of system
(\ref{eq4.2add}) gives
\begin{align}
\dot{U}_{0}(t)\leq &  -\rho \mathcal{\kappa}_{0}e^{\mathrm{T}}%
Qe+2\mathcal{\kappa}_{0}e^{\mathrm{T}}Q\xi_{3}\nonumber \\
\leq &  -\rho \mathcal{\kappa}_{0}e^{\mathrm{T}}Qe+\frac{\rho}{2}%
\mathcal{\kappa}_{0}e^{\mathrm{T}}Qe+\frac{2}{\rho}\mathcal{\kappa}_{0}\xi
_{3}^{\mathrm{T}}Q\xi_{3}\nonumber \\
=  &  -\frac{\rho}{2}U_{0}(t)+\frac{2}{\rho}\mathcal{\kappa}_{0}\xi
_{3}^{\mathrm{T}}Q\xi_{3}, \label{eq4.7}%
\end{align}
where we have used (\ref{eq4.5}) and the Young's inequality. It follows from
(\ref{eq4.add6}) that
\begin{align}
\xi_{3}^{\mathrm{T}}Q\xi_{3}\  &  \leq q\sum \limits_{i=1}^{q}(x^{\mathrm{T}%
}(\phi_{i})+e^{\mathrm{T}}(\phi_{i})){P(\phi_{i})B}B_{i}^{\mathrm{T}}%
QB_{i}{B^{\mathrm{T}}P(\phi_{i})}(x(\phi_{i})+e(\phi_{i}))\nonumber \\
&  \leq qn\left \Vert Q\right \Vert \sum \limits_{i=1}^{q}\left \Vert
B_{i}\right \Vert ^{2}{\gamma(\phi_{i})}(x^{\mathrm{T}}(\phi_{i})+e^{\mathrm{T}%
}(\phi_{i})){P(\phi_{i})}(x(\phi_{i})+e(\phi_{i}))\nonumber \\
&  \leq2qn\left \Vert Q\right \Vert \sum \limits_{i=1}^{q}\left \Vert
B_{i}\right \Vert ^{2}{\gamma(\phi_{i})}x^{\mathrm{T}}(\phi_{i}){P(\phi_{i}%
)}x(\phi_{i})+2qn\left \Vert Q\right \Vert \sum \limits_{i=1}^{q}\left \Vert
B_{i}\right \Vert ^{2}{\gamma(\phi_{i})}e^{\mathrm{T}}(\phi_{i}){P(\phi_{i}%
)}e(\phi_{i})\nonumber \\
&  \leq2qn\left \Vert Q\right \Vert \sum \limits_{i=1}^{q}\left \Vert
B_{i}\right \Vert ^{2}{\gamma(\phi_{i})}x^{\mathrm{T}}(\phi_{i}){P(\phi_{i}%
)}x(\phi_{i})+2\mu_{2}qn\left \Vert Q\right \Vert \left \Vert Q^{-1}\right \Vert
\sum \limits_{i=1}^{q}\left \Vert B_{i}\right \Vert ^{2}{\gamma}^{2}{(\phi_{i}%
)}e^{\mathrm{T}}(\phi_{i})Qe(\phi_{i})\nonumber \\
&  \leq2qn\left \Vert Q\right \Vert \sum \limits_{i=1}^{q}\left \Vert
B_{i}\right \Vert ^{2}\sum \limits_{i=1}^{q}{\gamma(\phi_{i})}x^{\mathrm{T}%
}(\phi_{i}){P(\phi_{i})}x(\phi_{i})+\frac{2}{\mathcal{\kappa}_{0}}\mu
_{2}qn\left \Vert Q\right \Vert \left \Vert Q^{-1}\right \Vert \sum \limits_{i=1}%
^{q}\left \Vert B_{i}\right \Vert ^{2}\sum \limits_{i=1}^{q}{\gamma}^{2}%
{(\phi_{i})}U_{0}(\phi_{i}), \label{eq4add7}%
\end{align}
where the first inequality holds due to the Jensen's inequality, the second
inequality holds due to (\ref{pleP2}), the third inequality holds due to the
Young's inequality, and the forth inequality holds due to (\ref{pleP6}).
Substituting (\ref{eq4add7}) into (\ref{eq4.7}) gives%
\begin{equation}
\dot{U}_{0}(t)\leq-\frac{\rho}{2}U_{0}(t)+\epsilon_{1}\sum \limits_{i=1}%
^{q}{\gamma(\phi_{i})}x^{\mathrm{T}}(\phi_{i}){P(\phi_{i})}x(\phi_{i}%
)+\frac{4}{\rho}\mu_{2}qn\left \Vert Q\right \Vert \left \Vert Q^{-1}\right \Vert
\sum \limits_{i=1}^{q}\left \Vert B_{i}\right \Vert ^{2}\sum \limits_{i=1}%
^{q}{\gamma}^{2}{(\phi_{i})}U_{0}(\phi_{i}), \label{eq4add8}%
\end{equation}
with $\epsilon_{1}=\frac{4}{\rho}\mathcal{\kappa}_{0}qn\left \Vert Q\right \Vert
\mathit{\Sigma}_{i=1}^{q}\left \Vert B_{i}\right \Vert ^{2}$.

Consider another function $U_{1}(t)=V_{0}(t)$, where $V_{0}(t)$ is defined in
(\ref{eq3.7}). Similar to (\ref{eq3.8}), we have
\begin{equation}
\dot{U}_{1}(t)\leq-\frac{\gamma}{2}U_{1}(t)+\frac{2}{\gamma}\mathit{\Xi
}^{\mathrm{T}}(t)P(t)\mathit{\Xi}(t). \label{eq4.9}%
\end{equation}
Notice that
\begin{equation}
\mathit{\Xi}^{\mathrm{T}}P(t)\mathit{\Xi}\leq2\mathit{\Delta}^{\mathrm{T}%
}P(t)\mathit{\Delta}+2\xi_{4}^{\mathrm{T}}P(t)\xi_{4}, \label{eq4.10}%
\end{equation}
where we have used Young's inequality and the definition%
\[
\xi_{4}(t)=\sum \limits_{i=1}^{q}B_{i}B^{\mathrm{T}}P{(\phi_{i})}e(\phi_{i}).
\]
It is easy to deduce that%
\begin{align}
\xi_{4}^{\mathrm{T}}P(t)\xi_{4}  &  \leq q\sum \limits_{i=1}^{q}e^{\mathrm{T}%
}(\phi_{i})P{(\phi_{i})}BB_{i}^{\mathrm{T}}P(t)B_{i}B^{\mathrm{T}}P{(\phi
_{i})}e(\phi_{i})\nonumber \\
&  \leq q\mu_{2}\gamma \sum \limits_{i=1}^{q}e^{\mathrm{T}}(\phi_{i})P{(\phi
_{i})}BB_{i}^{\mathrm{T}}B_{i}B^{\mathrm{T}}P{(\phi_{i})}e(\phi_{i}%
)\nonumber \\
&  \leq q\mu_{2}n\gamma \sum \limits_{i=1}^{q}\left \Vert B_{i}\right \Vert
^{2}{\gamma}(\phi_{i})e^{\mathrm{T}}(\phi_{i})P{(\phi_{i})}e(\phi
_{i})\nonumber \\
&  \leq q\mu_{2}^{2}n\left \Vert Q^{-1}\right \Vert \gamma \sum \limits_{i=1}%
^{q}\left \Vert B_{i}\right \Vert ^{2}{\gamma}^{2}(\phi_{i})e^{\mathrm{T}}%
(\phi_{i})Qe(\phi_{i})\nonumber \\
&  =\frac{q\mu_{2}^{2}n\left \Vert Q^{-1}\right \Vert }{\mathcal{\kappa}_{0}%
}\gamma \sum \limits_{i=1}^{q}\left \Vert B_{i}\right \Vert ^{2}{\gamma}^{2}%
{(\phi_{i})}U_{0}(\phi_{i}), \label{eq4.add10}%
\end{align}
where the first inequality is justified by the Jensen's inequality, the second
inequality follows from (\ref{pleP6}), the third inequality is derived from
(\ref{pleP2}), and the fourth inequality is supported by (\ref{pleP6}) again.

\qquad On the other hand, it follows from the denotation $\psi(t)=P(t)x(t)$
and (\ref{eq4.2}) that%
\begin{align*}
\dot{\psi}(t)= &  P(t)\dot{x}+\dot{P}(t)x\\
= &  P(t)\left(  Ax-\sum \limits_{i=1}^{q}B_{i}{B^{\mathrm{T}}P(\phi_{i}%
)}\left(  x\left(  \phi_{i}\right)  +e\left(  \phi_{i}\right)  \right)
\right)  +\dot{\gamma}\frac{\partial P(t)}{\partial \gamma}x\\
= &  \left(  \! \!P(t)A\!+\! \dot{\gamma}\frac{\partial P(t)}{\partial \gamma
}\right)  x\!-P(t)\sum \limits_{i=1}^{q}B_{i}{B^{\mathrm{T}}P(\phi_{i})}\left(
x\left(  \phi_{i}\right)  +e\left(  \phi_{i}\right)  \right)  .
\end{align*}
Similar to (\ref{eq3.9}), we have%
\begin{align*}
\mathit{\Delta}_{i}(t)= &  B_{i}B^{\mathrm{T}}\left(  \psi(t)-\psi(\phi
_{i})\right)
=    B_{i}B^{\mathrm{T}}\int_{t-\tau_{i}}^{t}\left(  \xi_{1}(s)-P(s)\xi
_{5}(s)\right)  \mathrm{d}s,\  \forall t\geq2\bar{\tau},
\end{align*}
where $\xi_{1}(t)$ is defined in (\ref{eq3.9}) and,
\[
\xi_{5}(t)=\sum \limits_{j=1}^{q}B_{j}B^{\mathrm{T}}P(\phi_{j})\left(
x(\phi_{j})+e\left(  \phi_{j}\right)  \right)  .
\]
By using the above relation, together with (\ref{eq3.5}), we deduce that
\[
\mathit{\Delta}^{\mathrm{T}}P(t)\mathit{\Delta}\leq q\mu_{2}\gamma
\sum \limits_{i\!=\!1}^{q}\! \! \left(  \!(1\!+\!k_{5})\beta_{1i}\!+\! \! \left(
1\!+\! \! \frac{1}{k_{5}}\right)  \! \! \beta_{3i}\right)  ,
\]
where $\!k_{5}>0$ is a constant to be determined, $\beta_{1i}$ is defined in
(\ref{eq3.10}), and
\[
\beta_{3i}=\! \left(  \int_{t\!-\! \tau_{i}}^{t}\! \!P(s)\xi_{5}(s)\mathrm{d}%
s\right)  ^{\mathrm{T}}\! \!BB_{i}^{\mathrm{T}}B_{i}B^{\mathrm{T}\!}\left(
\int_{t\!-\! \tau_{i}}^{t}\! \!P(s)\xi_{5}(s)\mathrm{d}s\right)  .
\]
It is easy to deduce that%
\begin{align*}
\beta_{3i}\leq &  \left \Vert B_{i}\right \Vert ^{2}\tau_{i}\int_{t-\tau_{i}%
}^{t}\xi_{5}^{\mathrm{T}}(s)P(s)BB^{\mathrm{T}}P(s)\xi_{5}(s)\mathrm{d}s\\
\leq &  \bar{\tau}n\left \Vert B_{i}\right \Vert ^{2}\int_{t-\bar{\tau}}%
^{t}\gamma(s)\xi_{5}^{\mathrm{T}}(s)P(s)\xi_{5}(s)\mathrm{d}s\\
\leq &  \bar{\tau}n\mu_{2}\left \Vert B_{i}\right \Vert ^{2}\int_{t-\bar{\tau}%
}^{t}\gamma^{2}(s)\left \Vert \xi_{5}(s)\right \Vert ^{2}\mathrm{d}s,
\end{align*}
where we have used the Jensen's inequality, the fact that $\tau_{i}\leq
\bar{\tau}$, and inequalities (\ref{pleP2}) and (\ref{pleP6}). By using the
Jensen's inequality and (\ref{pleP2}) again, we have%
\begin{align*}
\left \Vert \xi_{5}(s)\right \Vert ^{2}\leq &  2q\sum \limits_{j\!=\!1}%
^{q}x^{\mathrm{T}}(\phi_{j}(s))P(\phi_{j}(s))BB_{j}^{\mathrm{T}}%
B_{j}B^{\mathrm{T}}P(\phi_{j}(s))x(\phi_{j}(s))\\
&  +2q\sum \limits_{j\!=\!1}^{q}e^{\mathrm{T}}(\phi_{j}(s))P(\phi_{j}%
(s))BB_{j}^{\mathrm{T}}B_{j}B^{\mathrm{T}}P(\phi_{j}(s))e(\phi_{j}(s))\\
\leq &  2q\sum \limits_{j\!=\!1}^{q}\left \Vert B_{j}\right \Vert ^{2}%
x^{\mathrm{T}}(\phi_{j}(s))P(\phi_{j}(s))BB^{\mathrm{T}}P(\phi_{j}%
(s))x(\phi_{j}(s))\\
&  +2q\sum \limits_{j\!=\!1}^{q}\left \Vert B_{j}\right \Vert ^{2}e^{\mathrm{T}%
}(\phi_{j}(s))P(\phi_{j}(s))BB^{\mathrm{T}}P(\phi_{j}(s))e(\phi_{j}(s))\\
\leq &  2nq\sum \limits_{j\!=\!1}^{q}\left \Vert B_{j}\right \Vert ^{2}%
\gamma(\phi_{j}(s))x^{\mathrm{T}}(\phi_{j}(s))P(\phi_{j}(s))x(\phi_{j}(s))\\
&  +2nq\sum \limits_{j\!=\!1}^{q}\left \Vert B_{j}\right \Vert ^{2}\gamma
(\phi_{j}(s))e^{\mathrm{T}}(\phi_{j}(s))P(\phi_{j}(s))e(\phi_{j}(s))\\
\leq & 2nq\sum \limits_{j=1}^{q}\left \Vert B_{j}\right \Vert ^{2}\gamma(\phi
_{j}(s))U_{1}(\phi_{j}(s))+\frac{2nq\mu_{2}}{\mathcal{\kappa}_{0}}\left \Vert
Q^{-1}\right \Vert \sum \limits_{j\!=\!1}^{q}\left \Vert B_{j}\right \Vert
^{2}\gamma^{2}(\phi_{j}(s))U_{0}(\phi_{j}(s)),
\end{align*}
substitution of which into $\beta_{3i}$ gives
\begin{align*}
\beta_{3i}\leq &  2\bar{\tau}n^{2}\mu_{2}q\left \Vert B_{i}\right \Vert ^{2}%
\int_{t-\bar{\tau}}^{t}\sum \limits_{j\!=\!1}^{q}\! \! \left \Vert B_{j}%
\right \Vert ^{2}\gamma^{2}(s)\gamma(\phi_{j}(s))U_{1}(\phi_{j}(s))\mathrm{d}%
s\\
&  +\frac{2\bar{\tau}n^{2}\mu_{2}^{2}q\left \Vert Q^{-1}\right \Vert \left \Vert
B_{i}\right \Vert ^{2}}{\mathcal{\kappa}_{0}}\int_{t-\bar{\tau}}^{t}%
\sum \limits_{j\!=\!1}^{q}\! \! \left \Vert B_{j}\right \Vert ^{2}\gamma
^{2}(s)\gamma^{2}(\phi_{j}(s))U_{0}(\phi_{j}(s))\mathrm{d}s\\
\leq &  2\bar{\tau}n^{2}\mu_{2}q\left \Vert B_{i}\right \Vert ^{2}\int
_{t-\bar{\tau}}^{t}\sum \limits_{j=1}^{q}\left \Vert B_{j}\right \Vert ^{2}%
\gamma^{3}(\phi_{j}(s))U_{1}(\phi_{j}(s))\mathrm{d}s\\
&  +\frac{2\bar{\tau}n^{2}\mu_{2}^{2}q\left \Vert Q^{-1}\right \Vert \left \Vert
B_{i}\right \Vert ^{2}}{\mathcal{\kappa}_{0}}\int_{t-\bar{\tau}}^{t}%
\sum \limits_{j\!=\!1}^{q}\! \! \left \Vert B_{j}\right \Vert ^{2}\gamma^{4}%
(\phi_{j}(s))U_{0}(\phi_{j}(s))\mathrm{d}s\\
\leq &  2\bar{\tau}n^{2}\mu_{2}q\left \Vert B_{i}\right \Vert ^{2}%
\sum \limits_{j=\!1}^{q}\! \! \left \Vert B_{j}\right \Vert ^{2}\sum \limits_{j=1}%
^{q}\int_{t-\bar{\tau}}^{t}\gamma^{3}(\phi_{j}(s))U_{1}(\phi_{j}%
(s))\mathrm{d}s\\
&  +\frac{2\bar{\tau}n^{2}\mu_{2}^{2}q\gamma_{0}\left \Vert Q^{-1}\right \Vert
\left \Vert B_{i}\right \Vert ^{2}}{\mathcal{\kappa}_{0}}\sum \limits_{j=\!1}%
^{q}\! \! \left \Vert B_{j}\right \Vert ^{2}\sum \limits_{j\!=\!1}^{q}%
\! \! \int_{t-\bar{\tau}}^{t}\gamma^{3}(\phi_{j}(s))U_{0}(\phi_{j}%
(s))\mathrm{d}s,
\end{align*}
where we have used the fact that $\gamma(t)\leq \gamma \left(  \phi
_{j}(t)\right)  \leq \gamma_{0}$. In addition, Using the argument similar to
the proof of Theorem \ref{thm1} yileds
\begin{align*}
&  \sum \limits_{j=1}^{q}\int_{t-\bar{\tau}}^{t}\gamma^{3}(\phi_{j}%
(s))U_{1}(\phi_{j}(s))\mathrm{d}s\leq \frac{q}{1-d}\int_{t-2\bar{\tau}}%
^{t}\gamma^{3}(s)U_{1}(s)\mathrm{d}s,\\
&  \sum \limits_{j\!=\!1}^{q}\! \! \int_{t-\bar{\tau}}^{t}\gamma^{3}(\phi
_{j}(s))U_{0}(\phi_{j}(s))\mathrm{d}s\leq \frac{q}{1-d}\int_{t-2\bar{\tau}}%
^{t}\gamma^{3}(s)U_{0}(s)\mathrm{d}s.
\end{align*}
With this, the inequality in $\beta_{3i}$ can be further written as
\begin{align*}
\beta_{3i}\! \leq &  \! \frac{2\bar{\tau}n^{2}\mu_{2}q^{2}}{1\!-\!d}\left \Vert
B_{i}\right \Vert ^{2}\sum \limits_{j=\!1}^{q}\! \! \left \Vert B_{j}\right \Vert
^{2}\int_{t-2\bar{\tau}}^{t}\gamma^{3}(s)U_{1}(s)\mathrm{d}s  +\frac{2\bar{\tau}n^{2}\mu_{2}^{2}q^{2}\gamma_{0}}{\mathcal{\kappa}%
_{0}\left(  1\!-\!d\right)  }\left \Vert Q^{-1}\right \Vert \left \Vert
B_{i}\right \Vert ^{2}\sum \limits_{j=\!1}^{q}\! \! \left \Vert B_{j}\right \Vert
^{2}\int_{t-2\bar{\tau}}^{t}\gamma^{3}(s)U_{0}(s)\mathrm{d}s.
\end{align*}

Then it follows from (\ref{eq3.12}) and the above inequality of $\beta_{3i}$
that%
\begin{align*}
\mathit{\Delta}^{\mathrm{T}}P(t)\mathit{\Delta} &  \leq q\mu_{2}\gamma \left(
\left(  1+k_{5}\right)  \sum \limits_{i=1}^{q}\beta_{1i}(t)+\left(  1+\frac
{1}{k_{5}}\right)  \sum \limits_{i=1}^{q}\beta_{3i}(t)\right)  \\
&  =k_{6}q\mu_{2}\gamma \int_{t-2\bar{\tau}}^{t}\gamma^{3}(s)U_{1}\left(
s\right)  \mathrm{d}s+k_{7}q\mu_{2}\gamma \int_{t-2\bar{\tau}}^{t}\gamma
^{3}(s)U_{0}(s)\mathrm{d}s\\
&  \leq k_{8}q\mu_{2}\gamma \int_{t-2\bar{\tau}}^{t}\gamma^{3}(s)\left(
U_{0}(s)+U_{1}\left(  s\right)  \right)  \mathrm{d}s,
\end{align*}
where $k_{8}=\max \left \{  k_{6},k_{7}\right \}$ with
\begin{align*}
\frac{k_{6}}{2} &  \triangleq \left(  1+\frac{1}{k_{5}}\right)  \frac
{2\bar{\tau}n^{2}\mu_{2}q^{2}}{1\!-\!d}\left(  \sum \limits_{k=\!1}%
^{q}\! \! \left \Vert B_{k}\right \Vert ^{2}\right)  ^{2}=\left(  1+k_{5}\right)
2\bar{\tau}\sum \limits_{k=1}^{q}\left \Vert B_{k}\right \Vert ^{2}\left(
3n^{3}+\delta_{\mathrm{c}}k^{2}\gamma_{0}^{2/\mu-2}\right)  ,\\
k_{7} &  =\left(  1+\frac{1}{k_{5}}\right)  \frac{2\bar{\tau}n^{2}\mu_{2}%
^{2}q^{2}\gamma_{0}}{\mathcal{\kappa}_{0}\left(  1\!-\!d\right)  }\left \Vert
Q^{-1}\right \Vert \left(  \sum \limits_{k=\!1}^{q}\! \! \left \Vert B_{k}%
\right \Vert ^{2}\right)  ^{2}.
\end{align*}

Subsequently, substituting (\ref{eq4.10}), (\ref{eq4.add10}) and the above
inequality of $\mathit{\Delta}^{\mathrm{T}}P(t)\mathit{\Delta}$ into
(\ref{eq4.9}) gives
\begin{align}
\dot{U}_{1}(t)\leq &  -\frac{\gamma}{2}U_{1}(t)+4k_{8}q\mu_{2}\int
_{t-2\bar{\tau}}^{t}\gamma^{3}(s)\left(  U_{0}(s)+U_{1}(s)\right)
\mathrm{d}s+\frac{4q\mu_{2}^{2}n\left \Vert Q^{-1}\right \Vert }{\mathcal{\kappa
}_{0}}\sum \limits_{k=\!1}^{q}\! \! \left \Vert B_{k}\right \Vert ^{2}{\gamma}%
^{2}{(\phi_{i})}U_{0}(\phi_{i})\nonumber \\
\leq &  -\frac{\gamma}{2}U_{1}(t)+4k_{8}q\mu_{2}\int_{t-2\bar{\tau}}^{t}%
\gamma^{3}(s)\left(  U_{0}(s)+U_{1}(s)\right)  \mathrm{d}s+\frac{4q\mu_{2}%
^{2}n\left \Vert Q^{-1}\right \Vert }{\mathcal{\kappa}_{0}}\sum \limits_{k=1}%
^{q}\left \Vert B_{k}\right \Vert ^{2}\sum \limits_{i=1}^{q}{\gamma}^{2}%
{(\phi_{i})}U_{0}(\phi_{i}),\label{eq4.11r1}%
\end{align}
where we have used $\dot{\gamma}(t)<0,$ $\forall t\geq0$. It follows from
(\ref{eq4add8}) and (\ref{eq4.11r1}) that
\begin{align}
\dot{U}_{0}(t)+\dot{U}_{1}(t)\leq &  4k_{8}q\mu_{2}\int_{t-2\bar{\tau}}%
^{t}\gamma^{3}(s)\left(  U_{0}(s)+U_{1}\left(  s\right)  \right)
\mathrm{d}s+\epsilon_{1}\sum \limits_{i=1}^{q}{\gamma \left(  \phi_{i}\right)
}U_{1}(\phi_{i})\nonumber \\
&  -\frac{\rho}{2}U_{0}(t)+\epsilon_{2}\sum \limits_{i=1}^{q}{\gamma}%
^{2}{\left(  \phi_{i}\right)  }U_{0}(\phi_{i})-\frac{\gamma}{2}U_{1}%
(t),\label{eq4.add1}%
\end{align}
with
\[
\epsilon_{2}=4\left(  \frac{\left \Vert Q\right \Vert }{\rho}+\frac{\mu_{2}%
}{\mathcal{\kappa}_{0}}\right)  q\mu_{2}n\left \Vert Q^{-1}\right \Vert
\sum \limits_{k=\!1}^{q}\! \! \left \Vert B_{k}\right \Vert ^{2},
\]

which prompts us to choose the following Lyapunov-Krasovskii-like functional
\begin{equation}
U(t,x_{t},e_{t})=U_{0}(t)+U_{1}(t)+U_{2}(t,e_{t})+U_{3}(t,x_{t})+U_{4}%
(t,x_{t}), \label{eq4.12}%
\end{equation}
where%
\begin{align*}
U_{2}(t,e_{t})=  &  \frac{\epsilon_{2}}{1-d}\sum \limits_{i=1}^{q}\int
_{t-\tau_{i}}^{t}\left(  2+\frac{s-t}{\bar{\tau}}\right)  \gamma^{2}%
(s)U_{0}(s)\mathrm{d}s,\\
U_{3}(t,x_{t})=  &  \mathcal{\kappa}_{3}\int_{t-2\bar{\tau}}^{t}\left(
3+\frac{s-t}{\bar{\tau}}\right)  \gamma^{3}(s)\left(  U_{0}(s)+U_{1}%
(s)\right)  \mathrm{d}s,\\
U_{4}(t,x_{t})=  &  \frac{\epsilon_{1}}{1-d}\sum \limits_{i=1}^{q}\int
_{t-\tau_{i}}^{t}\left(  2+\frac{s-t}{\bar{\tau}}\right)  \gamma
(s)U_{1}(s)\mathrm{d}s,
\end{align*}
with $\mathcal{\kappa}_{3}>0$ being a constant to be designed.

The time derivative of $U_{2}(t,e_{t})$ can be calculated as%
\begin{align}
\dot{U}_{2}(t,e_{t})=  &  \frac{2\epsilon_{2}q}{1-d}\gamma^{2}U_{0}%
(t)-\frac{\epsilon_{2}}{(1-d)\bar{\tau}}\sum \limits_{i=1}^{q}\int_{t-\tau_{i}%
}^{t}\gamma^{2}(s)U_{0}(s)\mathrm{d}s-\sum \limits_{i=1}^{q}\frac{\epsilon
_{2}(1-\dot{\tau}_{i})}{1-d}\left(  2-\frac{\tau_{i}}{\bar{\tau}}\right)
\gamma^{2}(\phi_{i})U_{0}(\phi_{i})\nonumber \\
\leq &  \frac{2\epsilon_{2}q\gamma^{2}}{1-d}U_{0}(t)-\frac{1}{2\bar{\tau}%
}U_{2}(t,e_{t})-\epsilon_{2}\sum \limits_{i=1}^{q}\gamma^{2}(\phi_{i}%
)U_{0}(\phi_{i}), \label{eq4.add12}%
\end{align}
where we have used Assumption \ref{ass3}. In addition, we have%
\begin{align}
\dot{U}_{3}(t,x_{t})=  &  3\mathcal{\kappa}_{3}\gamma^{3}\left(
U_{0}(t)+U_{1}(t)\right)  -\frac{\mathcal{\kappa}_{3}}{\bar{\tau}}%
\int_{t-2\bar{\tau}}^{t}\gamma^{3}(s)\left(  U_{0}(s)+U_{1}(s)\right)
\mathrm{d}s,\nonumber \\
&  -\mathcal{\kappa}_{3}\gamma^{3}(t-2\bar{\tau})\left(  U_{0}\left(
t-2\bar{\tau}\right)  +U_{1}\left(  t-2\bar{\tau}\right)  \right),\label{eq4add13}%
\end{align}
and
\begin{align}
\dot{U}_{4}(t,x_{t})=  &  \frac{2\epsilon_{1}q}{1-d}\gamma U_{1}%
(t)-\frac{\epsilon_{1}}{(1-d)\bar{\tau}}\sum \limits_{i=1}^{q}\int_{t-\tau_{i}%
}^{t}\gamma(s)U_{1}(s)\mathrm{d}s-\sum \limits_{i=1}^{q}\frac{\epsilon
_{1}(1-\dot{\tau}_{i})}{1-d}\left(  2-\frac{\tau_{i}}{\bar{\tau}}\right)
\gamma(\phi_{i})U_{1}(\phi_{i})\nonumber \\
\leq &  \frac{2\epsilon_{1}q}{1-d}\gamma U_{1}(t)-\frac{1}{2\bar{\tau}}%
U_{4}(t,x_{t})-\epsilon_{1}\sum \limits_{i=1}^{q}\gamma(\phi_{i})U_{1}(\phi
_{i}). \label{eq4.add14}%
\end{align}
It follows from (\ref{eq4.add1}), (\ref{eq4.add12}), (\ref{eq4add13}) and
(\ref{eq4.add14}) that%
\begin{align}
\dot{U}(t,x_{t},e_{t})\leq &  -\frac{\rho}{2}U_{0}(t)+\epsilon_{2}%
\sum \limits_{i=1}^{q}{\gamma}^{2}{(\phi_{i})}U_{0}(\phi_{i})-\frac{\gamma}%
{2}U_{1}(t)+4k_{8}q\mu_{2}\int_{t-2\bar{\tau}}^{t}\gamma^{3}(s)\left(
U_{0}(s)+U_{1}(s)\right)  (s)\mathrm{d}s\nonumber \\
&  +\epsilon_{1}\sum \limits_{i=1}^{q}{\gamma(\phi_{i})}U_{1}(\phi_{i}%
)+\frac{2\epsilon_{2}q}{1-d}\gamma^{2}(t)U_{0}(t)-\epsilon_{2}\sum
\limits_{i=1}^{q}\gamma^{2}(\phi_{i})U_{0}(\phi_{i})-\frac{1}{2\bar{\tau}%
}U_{2}(t,e_{t})\nonumber \\
&  +3\mathcal{\kappa}_{3}\gamma^{3}\left(  U_{0}(s)+U_{1}(s)\right)
-\mathcal{\kappa}_{3}\gamma^{3}(t-2\bar{\tau})\left(  U_{0}(t-2\bar{\tau
})+U_{1}(t-2\bar{\tau})\right) \nonumber \\
&  -\frac{\mathcal{\kappa}_{3}}{\bar{\tau}}\int_{t-2\bar{\tau}}^{t}\gamma
^{3}(s)\left(  U_{0}(s)+U_{1}(s)\right)  (s)\mathrm{d}s+\frac{2\epsilon_{1}%
q}{1-d}\gamma U_{1}(t)-\epsilon_{1}\sum \limits_{i=1}^{q}\gamma(\phi_{i}%
)U_{1}(\phi_{i})-\frac{1}{2\bar{\tau}}U_{4}(t,x_{t})\nonumber \\
\leq &  -\left(  \frac{\rho}{2}-\frac{2\epsilon_{2}q}{1-d}\gamma
^{2}-3\mathcal{\kappa}_{3}\gamma^{3}\right)  U_{0}(t)-\frac{1}{2\bar{\tau}%
}U_{2}(t,e_{t})-\left(  \frac{\gamma}{2}-3\mathcal{\kappa}_{3}\gamma^{3}%
-\frac{2\epsilon_{1}q}{1-d}\gamma \right)  U_{1}(t)\nonumber \\
&  -\frac{1}{2\bar{\tau}}U_{4}(t,x_{t})-\left(  \frac{\mathcal{\kappa}_{3}%
}{\bar{\tau}}-4k_{8}q\mu_{2}\right)  \int_{t-2\bar{\tau}}^{t}\gamma
^{3}(s)\left(  U_{0}(s)+U_{1}(s)\right)  \mathrm{d}s. \label{eq4.13}%
\end{align}
Let $\mathcal{\kappa}_{0}$ be small enough such that $2\epsilon_{1}%
q/(1-d)\leq1/8$ and $\mathcal{\kappa}_{3}$ be large enough such that
$\mathcal{\kappa}_{3}\geq8k_{2}q\mu_{2}\bar{\tau}$. Then, it follows from
(\ref{eq4.13}) that%
\begin{align}
\dot{U}(t,x_{t},e_{t})\leq &  -\left(  \frac{\rho}{2}-\frac{2\epsilon_{2}%
q}{1-d}\gamma^{2}-3\mathcal{\kappa}_{3}\gamma^{3}\right)  U_{0}(t)-\left(
\frac{3}{8}\gamma-3\mathcal{\kappa}_{3}\gamma^{3}\right)  U_{1}(t)\nonumber \\
&  -\frac{1}{2\bar{\tau}}U_{2}(t,e_{t})-\frac{\mathcal{\kappa}_{3}}{2\bar
{\tau}}\int_{t-2\bar{\tau}}^{t}\gamma^{3}(s)\left(  U_{0}(s)+U_{1}(s)\right)
\mathrm{d}s-\frac{1}{2\bar{\tau}}U_{4}(t,x_{t})\nonumber \\
\leq &  -\left(  \frac{\rho}{2}-\frac{2\epsilon_{2}q}{1-d}\gamma
^{2}-3\mathcal{\kappa}_{3}\gamma^{3}\right)  U_{0}(t)-\left(  \frac{3}%
{8}\gamma-3\mathcal{\kappa}_{3}\gamma^{3}\right)  U_{1}(t)\nonumber \\
&  -\frac{1}{2\bar{\tau}}U_{2}(t,e_{t})-\frac{1}{6\bar{\tau}}U_{3}%
(t,x_{t})-\frac{1}{2\bar{\tau}}U_{4}(t,x_{t}). \label{eq4.14}%
\end{align}
Clearly, there exists a bounded constant $t_{2}\in \lbrack2\bar{\tau},\infty)$
independent of the initial condition such that%
\[
\frac{1}{4}\gamma(t)\leq \min \left \{  \frac{\rho}{2}-\frac{2\epsilon_{2}q}%
{1-d}\gamma^{2}-3\mathcal{\kappa}_{3}\gamma^{3},\frac{3}{8}\gamma
-3\mathcal{\kappa}_{3}\gamma^{3},\frac{1}{6\bar{\tau}}\right \}  ,
\]
is satisfied for all $t\geq t_{2}$, by which, it follows from (\ref{eq4.14})
that
\begin{equation}
\dot{U}(t,x_{t},e_{t})\leq-\frac{1}{4}\gamma(t)U(t,x_{t},e_{t}),\ t\geq t_{2}.
\label{eq4.15}%
\end{equation}
By using the comparison principle, we have from (\ref{eq4.15}) that
\begin{equation}
U(t,x_{t},e_{t})\leq \frac{f(t_{2})}{f(t)}U(t_{2},x_{t_{2}},e_{t_{2}}),\ t\geq
t_{2}, \label{eq4.16}%
\end{equation}
where $f(t)$ is defined in (\ref{eq3.add24}). In addition, similar to
(\ref{eq3.25}), we have
\begin{equation}
U(t,x_{t},e_{t})\geq \frac{\mu_{1}\gamma_{0}^{\delta}}{(\omega t+1)^{\delta \mu
}}\left \Vert x(t)\right \Vert ^{2}. \label{eq4.17}%
\end{equation}
It follows from (\ref{eq4.16}) and (\ref{eq4.17}) that%
\begin{equation}
\left \Vert x(t)\right \Vert ^{2}\leq \frac{f(t_{2})}{\mu_{1}\gamma_{0}^{\delta}%
}\  \frac{(\omega t+1)^{\delta \mu}}{f(t)}U(t_{2},x_{t_{2}},e_{t_{2}}),\ t\geq
t_{2}. \label{eq4.18}%
\end{equation}
Reviewing (\ref{eq4.12}), we have%
\begin{equation}
U(t,x_{t},e_{t})\geq \mathcal{\kappa}_{0}\lambda_{\min}(Q)\left \Vert
e(t)\right \Vert ^{2}. \label{eq4.19}%
\end{equation}
Then, it follows from (\ref{eq4.16}) and (\ref{eq4.19}) that
\begin{equation}
\left \Vert e(t)\right \Vert ^{2}\leq \frac{1}{\mathcal{\kappa}_{0}\lambda_{\min
}(Q)}\frac{f(t_{2})}{f(t)}U(t_{2},x_{t_{2}},e_{t_{2}}),\ t\geq t_{2}.
\label{eq4.20}%
\end{equation}
By employing the same argument as in the proof of Theorem \ref{thm1}, we
establish that, from (\ref{eq4.18}) and (\ref{eq4.20}), the zero solution of
the closed-loop system consisting of (\ref{sys}) and (\ref{eq4.1}) is
attractive. Furthermore, we can also infer that the zero solution of the
closed-loop system is stable in the sense of Lyapunov. Ultimately, it becomes
evident that (the zero solution of) the closed-loop system consisting of
(\ref{sys}) and (\ref{eq4.1}) is asymptotically stable. The proof is finished.
\end{proof}


\section{Simulation}

We consider a linear system with two time-varying input delays in the form of
(\ref{sys}) with $\tau_{1}(t)=1+0.4\sin \left(  2t\right)  $, $\tau
_{2}(t)=0.3+0.3\sin \left(  3t\right)  $, and
\begin{align*}
A  &  =\left[
\begin{array}
[c]{cccc}%
0 & 1 & 2 & -1\\
0 & 0 & -4 & 2\\
1 & 1 & 2 & -1\\
2 & 2 & 4 & -2
\end{array}
\right]  ,\ B_{1}=\left[
\begin{array}
[c]{cc}%
0 & 0\\
-1 & 1\\
0 & 0\\
0 & -1
\end{array}
\right]  , \ B_{2} =\left[
\begin{array}
[c]{cc}%
0 & 0\\
2 & -2\\
0 & 0\\
0 & 2
\end{array}
\right]  , \
C   =\left[
\begin{array}
[c]{cccc}%
0 & 0 & 1 & 0\\
1 & 0 & 0 & 0
\end{array}
\right]  ,
\end{align*}
It follows that $\lambda \left(  A\right)  =\left \{  0\right \}  $. Moreover, we
have%
\[
B=B_{1}+B_{2}=\left[
\begin{array}
[c]{cccc}%
0 & 1 & 0 & 0\\
0 & -1 & 0 & 1
\end{array}
\right]  ^{\mathrm{T}},
\]
which implies that $\left(  A,B\right)  $ is controllable. Thus Assumption
\ref{ass1} is fulfilled. In addition, both $\tau_{1}(t)$ and $\tau_{2}(t)$
satisfy Assumption \ref{ass3} with $\bar{\tau}=1.4$ and $d=0.9$. The initial
condition is set as $x(0)=\left[  1,2,-2,-6\right]  ^{\mathrm{T}}$ and
$u(\theta)=\left[  0,0\right]  ^{\mathrm{T}}, \theta \in \lbrack-2\bar{\tau},0)$.

\begin{figure}[ptb]
\centering
\includegraphics[width=0.5\textwidth]{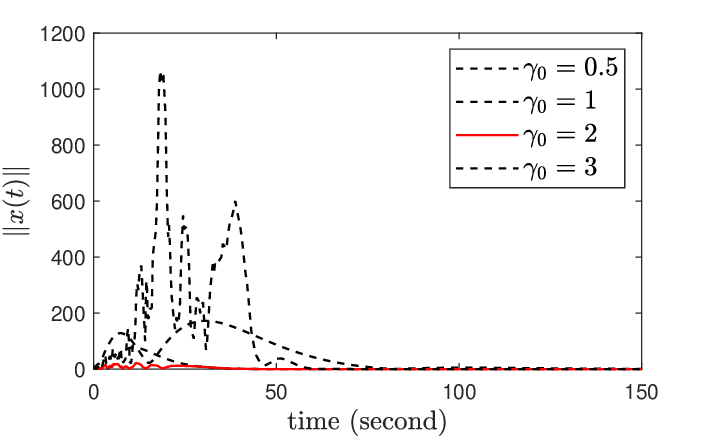}\caption{Norm of $x(t)$ with
different $\gamma_{0}$ ($\omega=1$ and $\mu=1/2$).}%
\label{fig1}%
\end{figure}

\begin{figure}[ptb]
\centering
\includegraphics[width=0.5\textwidth]{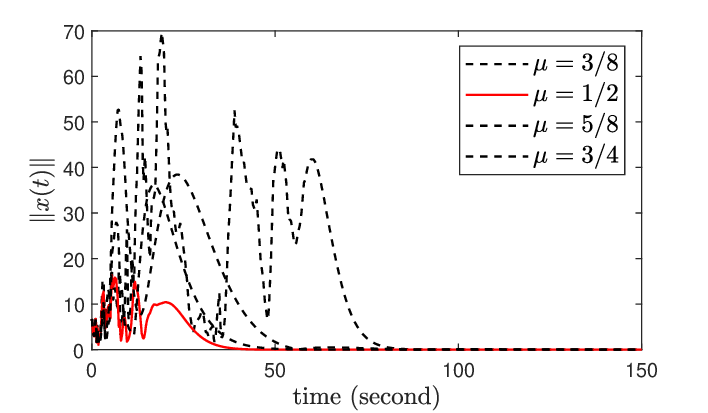}\caption{Norm of $x(t)$ with
different $\mu$ ($\gamma_{0}=2$ and $\omega=1$).}%
\label{fig2}%
\end{figure}

\begin{figure}[ptb]
\centering
\includegraphics[width=0.53\textwidth]{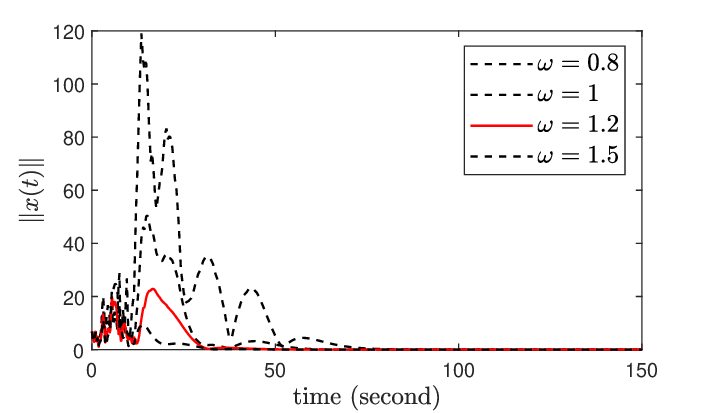}\caption{Norm of $x(t)$ with
different $\omega$ ($\gamma_{0}=2$ and $\mu=1/2$).}%
\label{fig3}%
\end{figure}

\begin{figure}[ptb]
\centering
\includegraphics[width=0.5\textwidth]{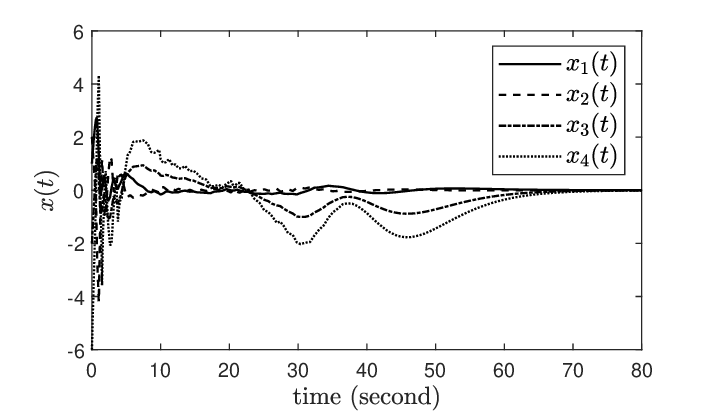}\caption{Controlled trajectories with
$\gamma_{0}=2$, $\mu=1/2$ and $\omega=1.2$.}%
\label{fig4}%
\end{figure}

\begin{figure}[ptb]
\centering
\includegraphics[width=0.5\textwidth]{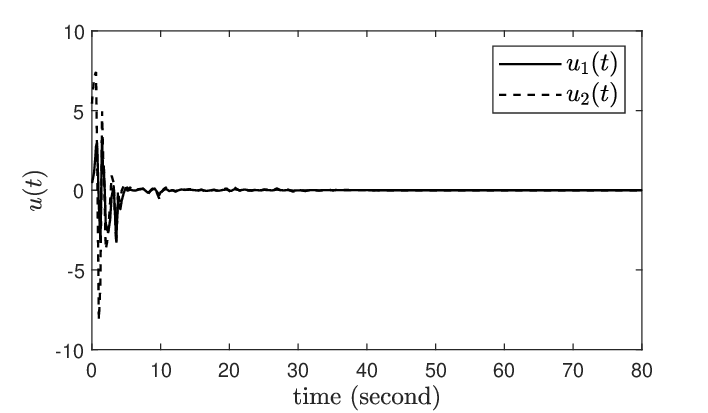}\caption{Control inputs with
$\gamma_{0}=2$, $\mu=1/2$, and $\omega=1.2$.}%
\label{fig5}%
\end{figure}

We first consider the case of state feedback. Using Theorem \ref{thm1}, we
design the following state feedback
\begin{align}
u\left(  t\right)  =  &  -Kx(t)=\!-\frac{1}{k_{0}}\! \left[
\begin{array}
[c]{cccc}%
k_{11} & k_{12} & -k_{13} & k_{14}\\
k_{21} & k_{22} & -k_{23} & k_{24}%
\end{array}
\right]  x\left(  t\right)  ,\  \nonumber \\
k_{0}=  &  25\gamma^{6}+80\gamma^{5}+148\gamma^{4}+128\gamma^{3}+96\gamma
^{2}+64,\nonumber \\
k_{11}=  &  \gamma^{4}\,{\left(  -5\gamma^{5}+3\gamma^{4}+24\gamma
^{3}+100\gamma^{2}+160\gamma+176\right)  }\nonumber \\
k_{12}=  &  \gamma^{3}\,{\left(  35\gamma^{4}+144\gamma^{3}+312\gamma
^{2}+352\gamma+224\right)  }\nonumber \\
k_{13}=  &  \gamma^{2}\, \left(  -5\gamma^{7}-12\gamma^{6}-2\gamma
^{5}+164\gamma^{4}   +512\gamma^{3}+848\gamma^{2}+704\gamma+384\right)  {,}\nonumber \\
k_{14}=  &  \gamma^{2}\,{\left(  15\gamma^{5}+106\gamma^{4}+280\gamma
^{3}+424\gamma^{2}+352\gamma+192\right)  ,}\nonumber \\
k_{21}=  &  2\gamma^{3}\,{\left(  5\gamma^{6}-3\gamma^{5}-20\gamma
^{4}-44\gamma^{3}+4\gamma^{2}+64\right)  ,}\nonumber \\
k_{22}=  &  2\gamma^{2}\,{\left(  -10\gamma^{5}-19\gamma^{4}-16\gamma
^{3}+36\gamma^{2}+64\gamma+96\right)  ,}\nonumber \\
k_{23}=  &  2\gamma \left(  5\gamma^{8}+12\gamma^{7}+61\gamma^{6}+138\gamma
^{5}   +248\gamma^{4}+216\gamma^{3}+288\gamma^{2}+192\gamma+256\right)
{,}\nonumber \\
k_{24}=  &  \gamma \left(  45\gamma^{6}+138\gamma^{5}+248\gamma^{4}%
+232\gamma^{3}   +288\gamma^{2}+192\gamma+256\right)  , \label{eq5.1}%
\end{align}
with $\gamma=\gamma(t)$ defined in (\ref{eqgammat}).

It follows from (\ref{eqgammat}) that the controller (\ref{eq5.1}) contains
three parameters $\gamma_{0}$, $\omega$, and $\mu$. Hence, we aim to assess
the impact of different values of $\gamma_{0}$, $\omega$, and $\mu$ on the
performance of the closed-loop system. Initially, we fix $\omega=1$ and
$\mu=1/2$, and then test four cases of $\gamma_{0}$: $0.5$, $1$, $2$, and $3$.
Figure \ref{fig1} displays the norm of states under these values, revealing
that the convergence time of the closed-loop system is shortest at $\gamma
_{0}=2$, followed by $\gamma_{0}=1$, and then $\gamma_{0}=0.5$. Conversely,
the longest convergence time occurs at $\gamma_{0}=3$. Next, keeping
$\gamma_{0}=2$ and $\omega=1$, we examine four cases of $\mu$: $3/8$, $1/2$,
$5/8$, and $3/4$. Figure \ref{fig2} illustrates the norm of states for each
case, demonstrating that the shortest convergence time is observed at
$\mu=1/2$, followed by $\mu=5/8$, and then $\mu=3/4$. Conversely, the longest
convergence time occurs at $\mu=3/8$. Finally, with $\gamma_{0}=2$ and
$\mu=1/2$, we explore four cases of $\omega$: $0.8$, $1$, $1.2$, and $1.5$.
Figure \ref{fig3} depicts the norm of states under these variations, revealing
that the shortest convergence time is observed at $\omega=1.2$, followed by
$\omega=1.5$, and then $\omega=1$. Conversely, the longest convergence time
occurs at $\omega=0.8$. Based on the aforementioned analysis, when $\gamma
_{0}=2$, $\mu=1/2$, and $\omega=1.2$, the system exhibits the shortest
convergence time. To provide further insight, we utilize Figure \ref{fig4} to
illustrate its state trajectories and Figure \ref{fig5} to depict its control
inputs, from which, we can see that the states can be quickly driven to $0$
within $70$s.

We now verify the effectiveness of Theorem \ref{thm2}. Consider the
observer-based controller (\ref{eq4.1}) where $K=B^{\mathrm{T}}P$ is well
defined in (\ref{eq5.1}) with $\gamma_{0}=1$, $\mu=0.95$ and $\omega=1$, and
$L$ is chosen as
\[
L=\left[
\begin{array}
[c]{cccc}%
1.5 & 1.75 & 2.5 & 5.25\\
1.5 & 1.75 & 1.5 & 3.25
\end{array}
\right]  ^{\mathrm{T}},
\]
which is such that $\lambda(A-LC)=\{-1,-1,-1,-1\}$. For simulation purposes,
the initial state of the observer is selected as ${\xi}%
(0)=[-1,1,1,-1]^{\mathrm{T}}$. The simulation result is shown in Fig.
\ref{fig6}, from which it is evident that the states gradually converge to
zero, thereby demonstrating the effectiveness of Theorem \ref{thm2}.
\begin{figure}[ptb]
\centering
\includegraphics[width=0.5\textwidth]{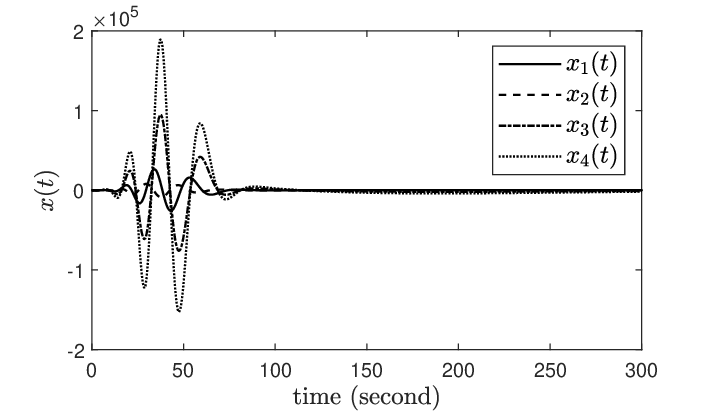}\caption{Controlled trajectories
with observer based controller}%
\label{fig6}%
\end{figure}

\section*{Appendix}

\subsection*{A1: Proof for the Inequality in (\ref{plePP3})}

Let Assumption \ref{ass1} be satisfied. Then, according to Theorem 1 in
\cite{zhou20auto}, for any given constant $\gamma_{0}>-2\phi(A) $, there
exists a constant $\bar{\delta}_{\mathrm{c}}\geq1$, which is independent of
$\gamma$, such that%
\begin{equation}
\frac{\mathrm{d}P(\gamma)}{\mathrm{d}\gamma}\leq \frac{\bar{\delta}%
_{\mathrm{c}}P(\gamma)}{\pi(\gamma) },\  \forall \gamma>\gamma_{0}. \label{0}%
\end{equation}
Subsequently, in order to prove the second inequality in (\ref{pleP5}), we
only need to demonstrate that
\[
\frac{\mathrm{d}P(\gamma)}{\mathrm{d}\gamma}\leq \frac{\delta_{\mathrm{c}%
}P(\gamma)}{\pi(\gamma) },\  \forall \gamma \in( -2\phi(A) ,\gamma_{0} ] ,
\]
where $\delta_{\mathrm{c}}$ is independent of $\gamma$. Our proof is similar
to the proof of Theorem 1 in \cite{zhou20auto} and is detailed here for
readers' convenience.

Taking derivative on both side of (\ref{2}) with respect to $\gamma$ yields%
\begin{equation}
-W=\left(  A+\frac{1}{2}\gamma I_{n}\right)  \frac{\mathrm{d}W}{\mathrm{d}%
\gamma}+\frac{\mathrm{d}W}{\mathrm{d}\gamma}\left(  A+\frac{1}{2}\gamma
I_{n}\right)  ^{\mathrm{T}}, \label{3}%
\end{equation}
and then multiplying both sides of (\ref{3}) by $W^{-\frac{1}{2}}$ from the
left and right sides, respectively, gives%
\begin{equation}
-I_{n}=\left(  -A_{1}-\frac{1}{2}\gamma I_{n}\right)  E+E\left(  -A_{1}%
-\frac{1}{2}\gamma I_{n}\right)  ^{\mathrm{T}}, \label{4}%
\end{equation}
where $E=-W^{-\frac{1}{2}}(\mathrm{d}W/\mathrm{d}\gamma)W^{-\frac{1}{2}}$ and
$A_{1}=W^{-\frac{1}{2}}AW^{\frac{1}{2}}$. Let $A_{2}=(-A_{1}-\gamma I_{n}/2)
/\pi(\gamma)$ and rewrite (\ref{4}) as
\begin{equation}
-I_{n}=A_{2}(\pi(\gamma) E) +(\pi( \gamma) E) A_{2}^{\mathrm{T}}. \label{7}%
\end{equation}
Notice that%
\begin{align*}
\operatorname{Re}\left(  \lambda_{i}\left(  -A_{1}-\frac{1}{2}\gamma
I_{n}\right)  \right)   =\operatorname{Re}(\lambda_{i}( -A_{1})) -\frac
{1}{2}\gamma
 <\operatorname{Re}(\lambda_{i}(-A_{1})) +\phi(A)
 \leq0,\forall \gamma>-2\phi(A) ,
\end{align*}
which implies that $A_{2}$ is Hurwitz for all $\gamma>-2\phi(A) $.
Subsequently, it follows from linear system theory that the Lyapunov equation
(\ref{7}) has a positive definite solution $\pi(\gamma) E>0\ $for all
$\gamma>-2\phi(A) $. We next show that $\left \Vert \pi(\gamma) E\right \Vert $
is upper bounded by a constant independent of $\gamma$.

As all eigenvalues of $\ A$ are identical and purely real, we have
$\mathrm{tr}^{2}(A)=n\mathrm{tr}(A^{2})$, by which, it is clear that%
\begin{align*}
 \mathrm{tr}\left(  \left(  A_{1}+\frac{1}{2}\gamma I_{n}\right)  \left(
A_{1}+\frac{1}{2}\gamma I_{n}\right)  ^{\mathrm{T}}\right)
=  &  \mathrm{tr}(A_{1}A_{1}^{\mathrm{T}})+\gamma \mathrm{tr}(A_{1})+\frac
{n}{4}\gamma^{2}\\
=  &  \mathrm{tr}(PAP^{-1}A^{\mathrm{T}})+\gamma \mathrm{tr}(A_{1})+\frac{n}%
{4}\gamma^{2}\\
\leq &  \frac{n-1}{2n}\pi^{2}(\gamma)+\frac{2}{n}\mathrm{tr}^{2}%
(A)-\mathrm{tr}(A^{2})+\gamma \mathrm{tr}(A)+\frac{n}{4}\gamma^{2}\\
=  &  \frac{n-1}{2n}\pi^{2}(\gamma)+\frac{2}{n}\mathrm{tr}^{2}(A)-\mathrm{tr}%
(A^{2})\\
&  +\frac{\pi(\gamma)-2\mathrm{tr}(A)}{n}\mathrm{tr}(A)+\frac{(\pi
(\gamma)-2\mathrm{tr}(A))^{2}}{4n}\\
=  &  \frac{2n-1}{4n}\pi^{2}(\gamma),
\end{align*}
where the first inequality holds due to (\ref{plePP6}). Thus we have
\begin{align}
\left \Vert A_{2}\right \Vert ^{2}  &  \leq \mathrm{tr}\left(  \left(
\frac{A_{1}+\frac{1}{2}\gamma I_{n}}{\pi(\gamma)}\right)  \left(  \frac
{A_{1}+\frac{1}{2}\gamma I_{n}}{\pi(\gamma)}\right)  ^{\mathrm{T}}\right)
\nonumber \\
&  =\frac{1}{\pi^{2}(\gamma)}\mathrm{tr}\left(  \left(  A_{1}+\frac{1}%
{2}\gamma I_{n}\right)  \left(  A_{1}+\frac{1}{2}\gamma I_{n}\right)
^{\mathrm{T}}\right) \nonumber \\
&  \leq \frac{2n-1}{4n},\  \gamma \in(-2\phi(A),\gamma_{0}]. \label{9}%
\end{align}
Now taking $\mathrm{vec}(\cdot)$ on both sides of (\ref{7}) gives%
\begin{align}
\mathrm{vec}(I_{n})  &  =-(I_{n}\otimes A_{2}+A_{2}\otimes I_{n}%
)\mathrm{vec}(\pi(\gamma)E)
  \triangleq-\mathit{\Phi}\mathrm{vec}(\pi(\gamma)E). \label{10}%
\end{align}
By (\ref{9}) we can see that%
\[
\left \Vert \mathit{\Phi}\right \Vert \leq \left \Vert I_{n}\otimes A_{2}%
\right \Vert +\left \Vert A_{2}\otimes I_{n}\right \Vert =2\left \Vert
A_{2}\right \Vert \leq2\sqrt{\frac{2n-1}{4n}},
\]
which implies that there exists a constant $c_{1}$ independent of $\gamma$,
such that%
\begin{equation}
\left \Vert \mathit{\Phi}^{\ast}\right \Vert \leq c_{1},\  \gamma \in
(-2\phi(A),\gamma_{0}], \label{11}%
\end{equation}
where $\mathit{\Phi}^{\ast}$ is the adjoint matrix of $\mathit{\Phi}$.

In addition, it follows from the definition of $A_{2}$ that
\[
\lambda_{i}(A_{2}) =\frac{-\lambda_{i}(A) -\frac{1}{2}\gamma}{\pi(\gamma) },
\]
by which, we can deduce that, for any $i\in \mathbf{I}_{1}^{n}$ and
$j\in \mathbf{I}_{1}^{n}$, there holds
\begin{align}
   \left \vert \lambda_{i}(A_{2}) +\lambda_{j}( A_{2}) \right \vert =\left \vert
\frac{-\lambda_{i}(A) -\lambda_{j}(A) -\gamma}{\pi(\gamma) }\right \vert
=    \frac{\gamma+2\phi(A) }{\pi(\gamma) }=\frac{1}{n},\  \  \gamma \in
(-2\phi(A) ,\gamma_{0} ] , \label{8}%
\end{align}
where we have used $\operatorname{Re}(\lambda_{i}(A) ) =\operatorname{Re}%
(\lambda_{j}(A)) =\phi(A) $. Subsequently, it follows from (\ref{11}) and
(\ref{8}) that there exists a constant $c_{2}>0$ independent of $\gamma$ such
that%
\begin{equation}
\left \Vert \mathit{\Phi}^{-1}\right \Vert = \frac{\left \Vert \mathit{\Phi
}^{\ast}\right \Vert }{\left \Vert \det( \mathit{\Phi}) \right \Vert }=
\frac{\left \Vert \mathit{\Phi}^{\ast}\right \Vert }{%
{\displaystyle \prod \limits_{i,j=1,2,\ldots,n}}
\left \vert \lambda_{i}(A_{2}) + \lambda_{j}( A_{2}) \right \vert }\leq c_{2}.
\label{13}%
\end{equation}
Then, it follows from (\ref{10}) and (\ref{13}) that
\begin{align}
\left \Vert \pi(\gamma)E\right \Vert  &  \leq \left \Vert \mathrm{vec}(\pi(\gamma)
E) \right \Vert \leq \left \Vert \mathit{\Phi}^{-1}\right \Vert \left \Vert
\mathrm{vec}(I_{n}) \right \Vert
 \leq \delta_{\mathrm{c}},\  \gamma \in(-2\phi(A) ,\gamma_{0} ] , \label{12}%
\end{align}
where $\delta_{\mathrm{c}}$ is a constant independent of $\gamma$. Finally, it
follows from (\ref{12}) and the definitions of $E$ and $\pi(\gamma)$ that
\[
\pi(\gamma) P^{-\frac{1}{2}}\frac{\mathrm{d}P(\gamma)}{\mathrm{d}\gamma
}P^{-\frac{1}{2}}= \pi(\gamma) E\leq \delta_{\mathrm{c}}I,\  \gamma \in(-2\phi(A)
,\gamma_{0} ] ,
\]
that is,%
\[
\frac{\mathrm{d}P(\gamma)}{\mathrm{d}\gamma}\leq \frac{\delta_{\mathrm{c}}}%
{\pi(\gamma) }P(\gamma),\  \gamma \in(-2\phi( A) ,\gamma_{0} ] ,
\]
which, together with (\ref{0}), establishes (\ref{plePP3}).

Finally, the proof for (\ref{plePP7}) is the same as those given in
\cite{zhou20auto} and is omitted. The proof is finished.

\subsection*{A2: A Proof for the Inequality in (\ref{pleP4})}

It follows from (\ref{pleP1}) that
\begin{align*}
  \frac{\mathrm{d}P}{\mathrm{d}\gamma}BB^{\mathrm{T}}\frac{\mathrm{d}%
P}{\mathrm{d}\gamma}=& \left(  \frac{\mathrm{d}P}{\mathrm{d}\gamma}\right)
^{\frac{1}{2}}\left(  \frac{\mathrm{d}P}{\mathrm{d}\gamma}\right)  ^{\frac
{1}{2}}BB^{\mathrm{T}}\left(  \frac{\mathrm{d}P}{\mathrm{d}\gamma}\right)
^{\frac{1}{2}}\left(  \frac{\mathrm{d}P}{\mathrm{d}\gamma}\right)  ^{\frac
{1}{2}}\\
\leq &  \left(  \frac{\mathrm{d}P}{\mathrm{d}\gamma}\right)  ^{\frac{1}{2}%
}\mathrm{tr}\left(  \left(  \frac{\mathrm{d}P}{\mathrm{d}\gamma}\right)
^{\frac{1}{2}}BB^{\mathrm{T}}\left(  \frac{\mathrm{d}P}{\mathrm{d}\gamma
}\right)  ^{\frac{1}{2}}\right)  \left(  \frac{\mathrm{d}P}{\mathrm{d}\gamma
}\right)  ^{\frac{1}{2}}\\
=  &  \left(  \frac{\mathrm{d}P}{\mathrm{d}\gamma}\right)  ^{\frac{1}{2}%
}\mathrm{tr}\left(  B^{\mathrm{T}}\frac{\mathrm{d}P}{\mathrm{d}\gamma
}B\right)  \left(  \frac{\mathrm{d}P}{\mathrm{d}\gamma}\right)  ^{\frac{1}{2}%
}\\
=  &  \frac{\mathrm{d}}{\mathrm{d}\gamma}\mathrm{tr}(B^{\mathrm{T}}PB)
\frac{\mathrm{d}P}{\mathrm{d}\gamma}\\
=  &  n\frac{\mathrm{d}P}{\mathrm{d}\gamma}.
\end{align*}
The proof is finished.

\subsection*{A3: A Proof for the Inequality in (\ref{eqadd1})}

Let $\phi_{i}(t)=t-\tau_{i}(t),i\in \mathbf{I}_{1}^{q}.$ It follows from
$\phi_{i}(0)=-\tau_{i}(0)\leq0,\  \phi_{i}(\bar{\tau})=\bar{\tau}-\tau_{i}%
(\bar{\tau})\geq0, $ and $\dot{\phi}_{i}(t)=1-\dot{\tau}_{i}(t)\geq1-d>0, $
that there exists a unique $T_{i}\in \lbrack0,\bar{\tau}]$ such that $\phi
_{i}(T_{i})=0.$ Without loss of generality, we assume that $0\leq T_{1}\leq
T_{2}\leq \cdots \leq T_{q}\leq \bar{\tau}. $ Then the closed-loop system can be
expressed by%
\[
\dot{x}\left(  t\right)  =\left \{
\begin{array}
[c]{l}%
Ax(t)+B\sum \limits_{i=1}^{q}u(t-\tau_{i}(t)),\quad0\leq t<T_{1}\\
Ax(t)+BK_{1}x(t-\tau_{1})+B\sum \limits_{i=2}^{q}u(t-\tau_{i}),\quad
t\in \mathcal{I}_{1}\\
\cdots \\
Ax(t)+B\sum \limits_{i=1}^{q-1}K_{i}x(t-\tau_{i})+Bu(t-\tau_{q}),\quad
t\in \mathcal{I}_{q-1}\\
Ax(t)+B\sum \limits_{i=1}^{q}K_{i}x(t-\tau_{i}),\quad T_{q}\leq t<\bar{\tau}\\
Ax(t)+B\sum \limits_{i=1}^{q}K_{i}x(t-\tau_{i}),\quad t\geq \bar{\tau}%
\end{array}
\right.
\]
where $\mathcal{I}_{i}=[T_{i},T_{i+1}),i \in \mathbf{I}_{1}^{q-1}$ and
$K_{i}=-B^{\mathrm{T}}P(\gamma(\phi_{i}(t)))$, $i\in \mathbf{I}_{1}^{q}$ are
uniformly bounded as $\lim_{t\rightarrow \infty}\gamma(t)=0.$ When $0\leq
t<T_{1},$ the closed-loop system is a linear system with the input $B
\mathit{\Sigma}_{i=1}^{q}u(t-\tau_{i}(t)).$ Thus, by the
variation-of-constants formula, we have%
\begin{align*}
\left \Vert x(t)\right \Vert  &  =\left \Vert \mathrm{e}^{At}x(0)+\int_{0}%
^{t}\mathrm{e}^{A(t-s)}B\sum \limits_{i=1}^{q}u(\phi_{i}(s))\mathrm{d}%
s\right \Vert \\
&  \leq k_{1}\left(  \left \Vert x(0)\right \Vert +\sup_{-\bar{\tau}\leq
\theta<0}\left \Vert u_{0}(\theta)\right \Vert \right)  ,\quad0\leq t<T_{1},
\end{align*}
where $k_{1}=k_{1}\left(  T_{1}\right)  >0$ is a constant. When $T_{1}\leq
t<T_{2},$ the closed-loop system is a linear time-delay system, and there
exists a constant $k_{2}>0$ such that (see pp. 141-143 in \cite{hale1977})%
\begin{align*}
\left \Vert x(t)\right \Vert &\leq k_{2}\left(  \sup_{-T_{1}\leq \theta \leq
0}\left \Vert x(T_{1}+\theta)\right \Vert +\sup_{-\bar{\tau}\leq \theta
<0}\left \Vert u_{0}(\theta)\right \Vert \right) \\
&  \leq k_{2}\left(  k_{1}\left(  \left \Vert x(0)\right \Vert +\sup_{-\bar
{\tau}\leq \theta<0}\left \Vert u_{0}(\theta)\right \Vert \right)  +\sup
_{-\bar{\tau}\leq \theta<0}\left \Vert u_{0}(\theta)\right \Vert \right) \\
&  \leq \kappa_{2}\left(  \left \Vert x(0)\right \Vert +\sup_{-\bar{\tau}%
\leq \theta<0}\left \Vert u_{0}(\theta)\right \Vert \right)  ,
\end{align*}
where $\kappa_{2}=\kappa_{2}(T_{2})>0$ is a constants. Repeating the above
process gives%
\[
\left \Vert x(t)\right \Vert \leq \kappa(T)\left(  \left \Vert x(0)\right \Vert
+\sup_{-\bar{\tau}\leq \theta<0}\left \Vert u_{0}(\theta)\right \Vert \right)
,\forall t\in \lbrack0,T],
\]
where $T>0$ is any constant. Thus the inequality (\ref{eqadd1}) holds true.

\end{document}